\documentclass{article}
\usepackage[utf8]{inputenc}

\usepackage[T1]{fontenc}
\usepackage[a4paper,top=5cm,left=3.75cm,right=3.75cm]{geometry}

\usepackage{graphicx}
\usepackage{tikz} 

\usetikzlibrary{positioning, shapes, arrows, calc, decorations.pathreplacing, arrows.meta, backgrounds, patterns, overlay-beamer-styles}

\usepackage{amsmath} 
\usepackage{amssymb}
\usepackage{amsthm}
\usepackage{fancyhdr}
\usepackage{hyperref}
\usepackage{algpseudocode}
\usepackage{algorithm2e}
\usepackage{soul}
\usepackage{placeins}

\usepackage{caption}
\usepackage{subcaption}
\usepackage{float}
\usepackage{bigints}

\newtheorem{lem}{Lemma}

\newtheorem{rem}{Remark}

\theoremstyle{definition}

\usepackage{multirow}
\usepackage{mathtools}
\usepackage{setspace}
\usepackage{tabularx}
\usepackage{array}
\usepackage{epigraph}
\usepackage{dsfont}
\usepackage[noabbrev]{cleveref}

\newcommand{\AK}[1]{\textcolor{black}{#1}}
\newcommand{\Ale}[1]{\textcolor{black}{#1}}
\newcommand{\F}{\mathbf{F}}

\definecolor{codegreen}{rgb}{0,0.6,0}
\definecolor{codegray}{rgb}{0.5,0.5,0.5}
\definecolor{codepurple}{rgb}{0.58,0,0.82}
\definecolor{backcolour}{rgb}{0.95,0.95,0.92}

\newcommand{\xddots}{%
  \raise 4pt \hbox {.}
  \mkern 6mu
  \raise 1pt \hbox {.}
  \mkern 6mu
  \raise -2pt \hbox {.}
}

\begin{document}
\def\N{{\mathbb N}}
\newcommand{\R}{\mathbb{R}}

\newcommand{\cve}[1]{\textcolor{blue}{\texttt{[VE: #1]}}}
\definecolor{auburn}{rgb}{0.43, 0.21, 0.1}
\newcommand{\fc}[1]{\textcolor{auburn}{\texttt{[FC: #1]}}}
\newcommand{\ve}[1]{\textcolor{blue}{#1}}

\fancyfoot[L]{\color{black}{\Large{C2 - Confidential}}}

\newcommand{\bu}{\boldsymbol{u}}
\newcommand{\bv}{\boldsymbol{v}}
\newcommand{\bphi}{\boldsymbol{\phi}}
\newcommand{\bvarphi}{\boldsymbol{\varphi}}
\newcommand{\bpsi}{\boldsymbol{\psi}}
\newcommand{\bxi}{\boldsymbol{\xi}}
\newcommand{\bzeta}{\boldsymbol{\zeta}}
\newcommand{\btheta}{\boldsymbol{\theta}}
\newcommand{\balpha}{\alpha}
\newcommand{\bw}{\boldsymbol{w}}
\newcommand{\bb}{\boldsymbol{b}}
\newcommand{\bn}{\boldsymbol{n}}
\newcommand{\bId}{\boldsymbol{\rm Id}}
\newcommand{\rL}{\mathrm L}
\newcommand{\rbP}{\boldsymbol{\mathrm P}}
\newcommand{\rS}{\mathrm S}
\newcommand{\bfx}{\boldsymbol{x}}
\newcommand{\f}{\mathbf{f}}

\newpage

\pagestyle{fancy}

\lhead{ }
\chead{}
\rhead{}
\lfoot{}
\cfoot{\thepage }
\rfoot{}
\begin{center}

\begin{huge} Optimal morphings for model-order reduction for poorly reducible problems with geometric variability \end{huge} \\

\begin{large} A. Kabalan$^{1,2}$, F. Casenave$^2$, F. Bordeu$^2$, V. Ehrlacher$^{1,3}$, A. Ern$^{1,3}$ \end{large} \\
$^1$ Cermics, Ecole nationale des ponts et chaussées, 6-8 Av. Blaise Pascal, Champs-sur-Marne, 77455 Marne-la-Vallée cedex 2, FRANCE,\\
$^2$ Safran Tech, Digital Sciences \& Technologies, Magny-Les-Hameaux, 78114, FRANCE, \\
$^3$ Inria Paris, 48 rue Barrault, CS 61534, 75647 Paris cedex, FRANCE. \\
 
\end{center}

\hrule

\vspace*{1pt}

\setstretch{1}

\paragraph{Abstract}    
We propose a new model-order reduction framework to poorly reducible problems arising from parametric partial differential equations with geometric variability. In such problems, the solution manifold exhibits a slowly decaying Kolmogorov $N$-width, \Ale{so that} standard projection-based model order reduction techniques based on linear subspace approximations \Ale{become ineffective}. To overcome this difficulty, we introduce an optimal morphing strategy: For each solution sample, we compute a bijective morphing from a reference domain to the sample domain such that, when all \Ale{the} solution fields are pulled back to the reference domain, their variability is reduced. We formulate a global optimization problem \Ale{on the morphings} that maximizes the energy captured by the first $r$ modes of the mapped fields obtained from Proper Orthogonal Decomposition, thus maximizing the reducibility of the dataset. Finally, using a non-intrusive Gaussian Process regression on the reduced coordinates, we build a fast surrogate model that can accurately predict new solutions, highlighting the practical benefits \Ale{of the proposed approach} for many-query applications. The framework is general, independent of the underlying partial differential equation, and applies to scenarios with either parameterized or non-parameterized geometries.

\section{Introduction}
\subsection{Background}

Proper Orthogonal Decomposition \Ale{(POD)} is a standard approach that extracts a reduced-order basis from high-fidelity snapshots. The efficiency \Ale{this approach} hinges on the assumption that the solution manifold $\mathcal{M}$, defined as the set of all possible solutions to a parametric partial differential equation (PDE), can be well-approximated by a low-dimensional linear space. While this assumption holds in many diffusion-dominated problems, there are important classes of PDEs for which it fails. In particular, transport-dominated or wave propagation problems with moving features exhibit \Ale{poor reducibility}. Recall that, \Ale{for a positive integer $r$,} the Kolmogorov $r$-width defined as
 $$
     d_r(\mathcal{M}) := \inf_{\underset{\rm dim(\mathcal{Z}_r)=r}{\mathcal{Z}_r} } \sup_{u \in \mathcal{M}} \inf_{v \in \mathcal{Z}_r} \|u - v\|_{\mathcal{Z}},
 $$
 characterizes the best-possible approximation of the manifold $\mathcal{M}$ by an $r$-dimensional linear space $\mathcal{Z}_r$. For transport problems, $d_r(\mathcal{M})$ decays only algebraically, implying that any reduced basis requires a large number of modes to achieve \Ale{high} accuracy. In practice, this \Ale{leads to} a very slow decay of POD eigenvalues, and consequently poor compression efficiency. 

Another challenge occurs with parameter-dependent domains, where each \Ale{snapshot} is defined on a different geometry. To address this, one introduces bijective morphing from a common reference domain, pulling back solutions onto the reference domain before applying POD.

While the problems of \Ale{poorly} reducibility and geometric \Ale{variability} can appear of different nature, both are amenable to solutions through morphings. In the first case,  dimensionality reduction can be \Ale{readily} applied, but morphings are necessary to obtain a good approximation space, whereas, in the second case, morphings are inevitable to achieve dimensionality reduction. \Ale{Hence, in both cases, given a family of snapshots $(u_i)_{1\leq i \leq n}$, one seeks a morphing family $(\bphi_i)_{1\leq i \leq n}$ so that the mapped family $(u_i\circ \bphi_i)_{1\leq i \leq n}$ has good reducibility properties.} In Figure \ref{fig: two mappings}, we show two different \Ale{morphings} from a reference geometry onto the same target geometry.

 \begin{figure}[ht] 
     \centering
     \includegraphics[scale=0.5]{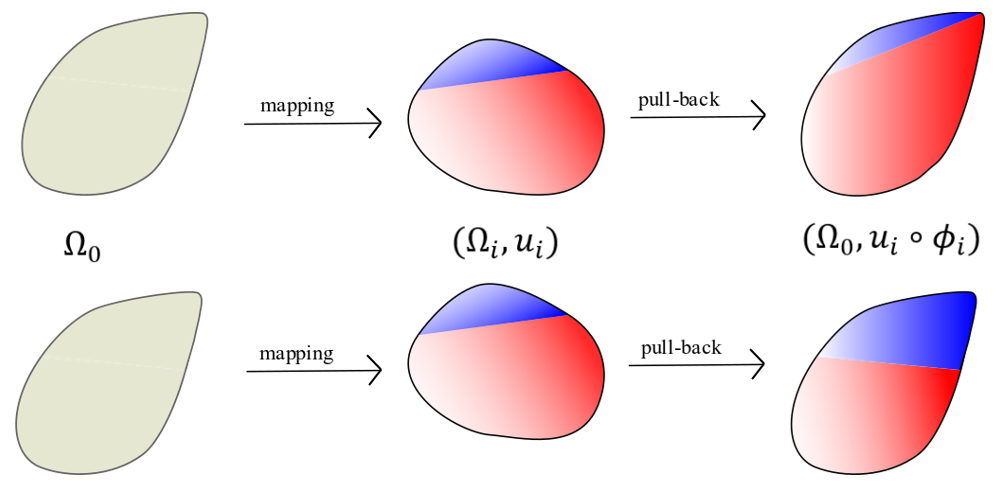}
\caption{ Example of two mappings from the same reference domain onto the same target domain.}
\label{fig: two mappings}
\end{figure}

 A natural question to ask is: \textbf{what are the best possible morphings, to enhance solution reducibility?} We directly address this question by formulating an optimization problem where the unknowns are the morphings $\{\boldsymbol{\phi}_i\}_{1\leq i\leq n}$, and the objective is to maximize the captured variance by the first $r$ POD modes of the transformed snapshots $\{u_i\circ \bphi_i\}_{1\leq i \leq n}$. In other words, we seek morphings that yield the largest possible reduction in dimensionality for a given target accuracy.

In this paper, we propose a new framework that explicitly formulates the search for morphings as a global optimization problem. Given a set of snapshots $(u_i)_{1\leq i\leq n}$ defined on domains $\{\Omega_i\}_{1\leq i\leq n}$, and \Ale{bijective morphings} $(\bphi_i : \Omega_0 \to \Omega_i)_{1\leq i\leq n}$, we define the pulled-back snapshots $(u_i \circ \bphi_i)_{1\leq i\leq n}$ on the reference domain $\Omega_0$. We then construct the correlation matrix
\begin{align} \label{C definition}
C[\Phi]_{ij} = \langle u_i \circ \bphi_i, \, u_j \circ \bphi_j \rangle_{\Omega_0}, \quad 1 \leq i,j \leq n,
\end{align}
and denote by $\lambda_1[\Phi] \geq \cdots \geq \lambda_n[\Phi]$ its eigenvalues. For a parameter $1\leq r \leq n$, our objective is to maximize the POD efficiency functional
\begin{align} \label{J definition}
 \displaystyle J_r[\Phi] = \frac{ \displaystyle\sum_{j=1}^r \lambda_j[\Phi]}{\mathrm{{\rm Tr}}(C[\Phi])}    
\end{align}
that is, to identify the morphing family $\Phi = \{\bphi_i\}_{1 \leq i \leq n}$ that maximizes the energy fraction captured by the first $r$ POD modes.
\subsection{Related works}
\subsubsection{Poorly-reducible problems}
In order to overcome the Kolmogorov barrier, researchers have proposed several strategies. Dictionary-based methods rely on constructing a collection of linear reduced-order models by partitioning different regions in the time/parameter space or in the solution space \cite{buhr2018randomized,nouy2024dictionary,daniel2020model,rapun2010reduced,dihlmann2011model,amsallem2012nonlinear}. Other approaches rely on Grasmaniann learning \cite{mosquera2021generalization,amsallem2008interpolation,mosquera2019pod,zimmermann2018geometric}, autoencoders to construct a low-dimensional embedding \cite{lee2020model,mojgani2021low,fresca2022pod,fresca2021comprehensive,cote2023hamiltonian}, and  neural networks \cite{barnett2023neural,mojgani2023kolmogorov}. Another popular approach that will be used in this \Ale{Thesis}, is to consider a transformation of the snapshots \Ale{to better exploit} dimensionality reduction. Common examples of \Ale{such} methods are freezing \cite{ohlberger2013nonlinear}, implicit feature tracking \cite{mirhoseini2023model}, convex displacement interpolation \cite{cucchiara2024model}, and registration \cite{taddei2020registration,taddei2023compositional,gowrachari2025model}. Other examples where \Ale{similar methods} were successfully applied can also be found in \cite{nair2019transported,welper2017interpolation,welper2020transformed,cagniart2018model}.

In this work, we pursue a complementary direction: We seek a family of \Ale{morphings} $\Phi :=\{\bphi_i\}_{1\leq i \leq n} $ such that \Ale{the mapped family} $\{u_i\circ\bphi_i\}_{1\leq i \leq n}$ is more amenable to data compression using linear compression methods.

\subsubsection{Morphing techniques}
When all the domains $\{\Omega_i\}_{1\leq i \leq n}$ share a common topology, the most common solution in the presence of geometric variabilities is to find, for each domain $\Omega_{i}$, a bijective morphing $\bphi_{i}$, from a fixed reference domain $\Omega_0$ onto $\Omega_{i}$, that is $\bphi_i(\Omega_{0})=\Omega_{i}$, 
and we can then pull back the solution to the reference domain as $u_i \circ \boldsymbol{\phi}_i$. This way, model order reduction techniques are applied on the transformed solution manifold $\mathcal{M}_{\Phi}:=\{u_{i}\circ \bphi_{i} \}_{1 \leq i \leq n}$. While finding such morphings is not straightforward, many researchers in the reduced-order modeling community and other related fields proposed methods to efficiently compute morphings to transform geometrical domains \cite{  sederberg1986free,baker2002mesh,beg2005computing,de2007mesh,masud2007adaptive,shontz2012robust,staten2012comparison,de2016optimization}, and how they can be applied to problems in reduced-order modeling \cite{casenave2024mmgp,ye2024data,sieger2014rbf,taddei2020registration,taddei2023compositional,rozza2008reduced,kabalan2025elasticity}.

\subsection{Contribution}
The main contributions of this paper are as follows. First, we introduce the global optimization problem for the computation of the optimal morphings, providing a systematic criterion for enhancing reducibility. Second, we design an gradient ascent algorithm to solve the optimization problem, employing \Ale{elasticity-based} inner products that regularize the updates. \Ale{Moreover, the functional $J_r$ is augmented by a penalty term based on the hyperelastic energy of the displacement fields associated to the morphings in order to preserve the invertibility of the morphings} and guarantee smooth deformations of the domains. Finally, after morphing optimization and POD reduction, we construct a non-intrusive surrogate model using Gaussian Process regression on the reduced coordinates, which enables fast prediction of new solutions.

The framework developed herein is general and independent of the underlying PDE. It addresses two major obstacles in projection-based model-order reduction: the slow decay of the Kolmogorov $N$-width and the intrinsic variability of geometric domains. 


\section{Methodology}
 In this section, we formulate the optimization problem to find the morphing family $\{\bphi_i\}_{1\leq i \leq n}$ and the building blocks to solve this optimization problem.

\subsection{ Problem setting and notation}

Let $\{\Omega_i\}_{1\leq i \leq n} \subset \mathbb{R}^d$  be a family of $n$ spatial domains, with $d=2$ or $3$, representing different geometric configurations in \Ale{the model} problem. Assume that all \Ale{the} domains share the same topology.  Let $\Omega_0$ be a chosen reference domain, homeomorphic to each $\Omega_i$.  $\Omega_0$ could be one of the actual domains (say $\Omega_1$) or a suitable reference shape. For all $1 \leq i\leq n$, let $u_i$ be a solution field $u_i:\Omega_i \to \mathbb{R}$ (for simplicity, we consider a scalar field; extension to vector fields is straightforward). Each $u_i$ is \Ale{meant to be} high-fidelity solution of a parametric PDE corresponding to some parameter $\mu_i \in \mathcal{P}$, i.e., $u_i = u(\cdot; \mu_i)$. We make no particular assumptions on the governing PDE here.\\

Our aim is to find a morphing family $\Phi := (\bphi_i)_{1 \leq i\leq n}$ so that \Ale{each morphing $\bphi_i$} transforms the reference domain $\Omega_0$ onto each target domain $\Omega_i$, i.e. $\bphi(\Omega_0)=\Omega_i$ for all $1\leq i \leq n$. The morphing $\bphi_i: \Omega_0 \to \Omega_i$ is required to be a bijective, smooth mapping (a diffeomorphism) satisfying $\bphi_i(\Omega_0) = \Omega_i$. While such morphings always exist in theory since the domains are topologically equivalent, they are not unique. We therefore have \Ale{given some} freedom to choose $\bphi_i$ to serve our purpose: \Ale{which} is finding $(\bphi_i)_{1\leq i \leq n}$ so that the family $\{u_i\circ\bphi_i\}_{1\leq i \leq n}$ can be well approximated using POD. Once the \Ale{family} $(\bphi_i)_{1 \leq i \leq n}$ is determined, we can \Ale{pullback} the solution fields $(u_i)_{1\leq i \leq n }$ to the reference domain by defining $u^{\rm ref}_i:\Omega_0 \to \mathbb{R}$ for all $1\leq i \leq n$ as 
\begin{align}\label{eq: u ref}
    u^{\rm ref}(\boldsymbol{x}) := u_i\circ \bphi_i(\boldsymbol{x}), \quad \forall \boldsymbol{x} \in \Omega_0.
\end{align}

Let $\displaystyle \langle f, g\rangle_{\Omega_0} := \int_{\Omega_0} f(\boldsymbol{x})g(\boldsymbol{x})\,d\boldsymbol{x}$ denote the \Ale{canonical} $L^2$-inner product on $\Omega_0$ (or a suitable weighted inner product relevant to the PDE), \Ale{and let $\|\cdot\|_{\Omega_0}$ denote the associated norm.}. Define $\mathbf{M}_i$ as the \Ale{set} of bijective morphings from $\Omega_0$ to $\Omega_i$ and let $\mathbf{M}:=\mathbf{M}_1 \times \mathbf{M}_2 \times \cdots \times \mathbf{M}_n$.\\
\Ale{Let $\Phi \in \mathbf{M}$}. We define the correlation matrix for the transformed snapshots, $C$, on $ \mathbf{M}$ as $C[\Phi] \in S_n^+$,  where $\mathcal S_n^+$ is the set of symmetric positive semi-definite matrices of size $n \times n$, such that
$$C[\Phi]_{ij} \;:=\; \langle\,u^{\rm ref}_i,\;u^{\rm ref}_j\,\rangle_{\Omega_0} , \quad 1\le i,j\le n.$$
We denote by $\lambda_1[\Phi]\ge \lambda_2[\Phi]\ge \cdots \ge \lambda_n[\Phi]\ge 0$, (counting multiplicity), the eigenvalues of $C[\Phi]$, and by $\{\displaystyle \zeta_1[\Phi] , \zeta_2[\Phi] , \cdots , \zeta_n[\Phi]\}\subset \mathbb{R}^n$ an orthonormal family of corresponding eigenvectors. The eigenvalue $\lambda_k[\Phi]$ corresponds to the energy captured by the $k$-th POD mode. \Ale{For all $1\leq i,k\leq n$, let $\zeta_{i,k}[\Phi]$ to be the $k$-correspondence of $\zeta_{i}$ in the canonical basis of $\R^n$.}

For a parameter $1\leq r\leq n$, our objective is to choose the family of morphing $\Phi$ so that the first $r$ POD modes capture as much as possible of the total energy. Thus, we select $\Phi_r^{\rm opt}$ so that 
\begin{align}\label{optimization statement}
    \Phi^{\rm opt}_r \in \arg \max_{\Phi \in \mathbf{M}} J_{ r}[\Phi],
\end{align} 
where the functional $\displaystyle J_r$ is as defined in \eqref{J definition}.  In other words, we seek morphings that maximize the energy fraction captured by the first $r$ modes.

The optimization problem \eqref{optimization statement} is non-trivial: \Ale{even after discretization}, the dimension of the search space $\mathbf{M}$ is very large, and $J_r[\Phi]$ is a non-convex functional. Thus, directly optimizing over all possible morphings is untractable. In what follows, we present the building blocks to solve \eqref{optimization statement}.

\subsection{Gradient algorithm}

\Ale{The first idea is} apply a gradient ascent method (for maximization) by iteratively updating the morphings in the direction of steepest ascent of the objective. \Ale{One important ingredient is the choice of the inner product to define the gradient, i.e.~the Riesz representative of the differential of the cost function.} In the present setting, an additional challenge is the need to \Ale{ensure that all the morphings remain} smooth and bijective at each iteration. 
\Ale{We will see in Section \ref{sec: bijectivity} that this entails adding an hyperelasticity penalty term to the cost functional $J_r$ defined in \eqref{J definition}.}

\subsubsection*{Differential of $J_r$}

We begin by formulating the first variation of the objective functional. Given a current guess $\Phi \in \mathbf{M}$ and a perturbation (variation) $\Psi = (\bpsi_i)_{1\leq i \leq n}$ in the set of admissible directions, the \textit{Gateaux derivative} of $J_r$ at $\Phi$ in the direction of $\Psi$ defined by:
\begin{align}\label{differetial of J def}
  D J_r[\Phi][\Psi] := \lim_{\epsilon\to 0} \frac{J_r[\Phi + \epsilon \Psi] - J_r[\Phi]}{\epsilon}\,.  
\end{align}
\begin{lem} \label{lem : diff J}
Let $\delta_{ij}$ the Kronecker delta. \Ale{Define $\mathbf{Z}_{ij}[\Phi][\Psi]:= \langle u_i \circ \bphi_i , (\Vec{\nabla} u_j \circ \bphi_j)\cdot \bpsi_j\rangle_{\Omega_0}$}. Let $\Phi\in \mathbf{M}$, and let $\Psi:= (\bpsi_i)_{1\leq i \leq n}$ a perturbation. The differential \Ale{of $J_r$ in the direction of $\Psi$} is computed explicitly as
    \begin{align}
    D J_r[\Phi][\Psi] &=\sum_{i=1}^n\sum_{j=1}^n \sum_{k=1}^r \left( \frac{ 2\zeta_{k,i}[\Phi]\zeta_{k,j}[\Phi] }{{\rm Tr} (C[\Phi]) } -\frac{ \displaystyle 2\lambda_k[\Phi]}{{\rm Tr} (C[\Phi])^2 }\delta_{ij}\right)\mathbf{Z}_{ji}[\Phi][\Psi]  .
\end{align}
\end{lem}
\begin{proof}

We have, for all $1\leq i \leq n$, and all $\boldsymbol{x} \in \Omega_0$
\begin{align*}
     \hspace{0.5em} u_i \circ (\bphi_i + \epsilon \bpsi_i)(\boldsymbol{x}) &= u_i(\bphi_i(\boldsymbol{x}) + \epsilon \bpsi_i(\boldsymbol{x}))\\
    &\simeq (u_i\circ\bphi_i)(\boldsymbol{x}) + \epsilon \left((\Vec{\nabla} u_i\circ\bphi_i)\cdot \bpsi_i\right)(\boldsymbol{x}),
\end{align*}
where we \Ale{neglected} higher-order terms \Ale{in $\epsilon$}. Next, we evaluate 
\begin{flalign}\label{Cij differential}
\displaystyle  C_{ij}[\Phi + \epsilon \Psi] & = \langle u_i \circ (\bphi_i + \epsilon \bpsi_i), u_j \circ (\bphi_j + \epsilon \bpsi_j) \rangle_{\Omega_0} &&\nonumber\\
 &= \int_{\Omega_0} \left(u_i \circ (\bphi_i + \epsilon \bpsi_i)\right)(\boldsymbol{x})  (u_j \circ (\bphi_j + \epsilon \bpsi_j))(\boldsymbol{x}) d\boldsymbol{x} &&\nonumber\\
 &=\int_{\Omega_0} (u_i \circ \bphi_i)(\boldsymbol{x}) ( u_j \circ \bphi_j)(\boldsymbol{x}) d\boldsymbol{x} + \epsilon\int_{\Omega_0}(u_i \circ \bphi_i)(\boldsymbol{x})  \left(\Vec{\nabla} u_j \circ \bphi_j\right)(\boldsymbol{x})\cdot \bpsi_j(\boldsymbol{x})d\boldsymbol{x} &\nonumber\\
 &{\color{white}=}+\epsilon\int_{\Omega_0}  \left(\Vec{\nabla} u_i\circ \bphi_i\right)(\boldsymbol{x})\cdot \bpsi_i(\boldsymbol{x}) ( u_j \circ \bphi_j)(\boldsymbol{x}) d\boldsymbol{x} +  o(\epsilon)\nonumber.&&
\end{flalign}
This proves that $D C_{ij}[\Phi][\Psi] = \mathbf{Z}_{ij}[\Phi][\Psi] + \mathbf{Z}_{ji}[\Phi][\Psi]$, and we set $D C[\Phi][\Psi]=(D C_{ij}[\Phi][\Psi])_{1\leq i,j\leq n}$.

 \Ale{We now} evaluate \AK{the differential of $\lambda_k[\Phi]$ in the direction of $\Psi$,} denoted as $D \lambda_k[\Phi][\Psi]$. We have, for all $1\leq k \leq n$, 
\begin{align*}
    ||\zeta_k[\Phi] ||^2=1.
\end{align*}
By taking the differential \Ale{of this identity in the direction of $\Psi$}, we get 
\begin{align*}
    \langle D \zeta_k[\Phi][\Psi] , \zeta_k[\Phi] \rangle = 0.
\end{align*}
Differentiating now the identity $C[\Phi]\zeta_k[\Phi]= \lambda_k[\Phi] \zeta_k[\Phi] $, we obtain 
\begin{align*}
    D C[\Phi][\Psi] \zeta_k[\Phi] + C[\Phi]D \zeta_k[\Phi][\Psi] = D \lambda_k[\Phi][\Psi]  \zeta_k[\Phi] + \lambda_k[\Phi] D \zeta_k[\Phi][\Psi].
\end{align*}
We \Ale{take the inner product of} the last equation with $\zeta_k[\Phi]$
\begin{align*}
   (\zeta_k[\Phi])^TD C[\Phi][\Psi] \zeta_k[\Phi] + 0 = D \lambda_k[\Phi][\Psi] +0 .
\end{align*}
Thus, we have for all $1\leq k \leq n$
\begin{align*}
    \lambda_k[\Phi + \epsilon \Psi][\Psi] = \lambda_k[\Phi][\Psi] + (\zeta_k[\Phi])^{\mathrm T} DC[\Phi][\Psi] \zeta_k[\Phi].
\end{align*}
Taking the sum over the first $r$ eigenvalues, we obtain: 
\begin{align}
    \displaystyle \sum_{j=1}^r \lambda_j[\Phi + \epsilon \Psi][\Psi] = \displaystyle \sum_{j=1}^r \lambda_j[\Phi][\Psi] + \epsilon {\rm Tr}( (\mathrm{Z}_r[\Phi])^T D C[\Phi][\Psi] \mathrm{Z}_r[\Phi] ).
\end{align}
with $\mathrm{Z}_r[\Phi]= (\zeta_1[\Phi], \zeta_2[\Phi], \cdots, \zeta_r[\Phi] )^T$. Now we can evaluate \eqref{differetial of J def} to obtain 
\begin{align}\label{differential J general expression}
     D J_r[\Phi][\Psi] =  \frac{{\rm Tr}( (\mathrm{Z}_r[\Phi])^T D C[\Phi][\Psi] \mathrm{Z}_r[\Phi] )}{{\rm Tr} (C[\Phi]) } - \frac{ \displaystyle \sum_{k=1}^r \lambda_k[\Phi]}{{\rm Tr} (C[\Phi])^2 } \times {\rm Tr} (D C[\Phi][\Psi]) ,
\end{align}
which can be written explicitly as
\begin{flalign*}
    D J_r[\Phi][\Psi] &= \frac{2}{{\rm Tr}(C[\Phi])}\sum_{i=1}^n \sum_{j=1}^n \sum_{k=1}^r  \zeta_{k,i}[\Phi] \zeta_{k,j}[\Phi] \displaystyle \int_{\Omega_0} (u_j \circ \bphi_j)(\boldsymbol{x}) ( \Vec{\nabla} u_i \circ \bphi_i)(\boldsymbol{x})\cdot \bpsi_i(\boldsymbol{x})d\boldsymbol{x} &&\\
    &- \frac{ \displaystyle 2\sum_{k=1}^r \lambda_k[\Phi]}{{\rm Tr} (C[\Phi])^2 } \sum_{i=1}^n \int_{\Omega_0} (u_j \circ \bphi_j)(\boldsymbol{x}) ( \Vec{\nabla} u_i \circ \bphi_i)(\boldsymbol{x})\cdot \bpsi_i(\boldsymbol{x})d\boldsymbol{x} \\
    &=\sum_{i=1}^n \left( \sum_{j=1}^n \sum_{k=1}^r  \frac{ 2\zeta_{k,i}[\Phi]\zeta_{k,j}[\Phi] }{Tr (C_[\Phi]) }\mathbf{Z}_{ji}[\Phi][\Psi] -  \sum_{k=1}^r \frac{ \displaystyle 2 \lambda_k[\Phi]}{Tr (C_[\Phi])^2 } \mathbf{Z}_{ji}[\Phi][\Psi] \right) .\\
\end{flalign*} 
\AK{We can rewrite the term }
\begin{align*}
    &\sum_{j=1}^n \sum_{k=1}^r  \frac{ 2\zeta_{k,i}[\Phi]\zeta_{k,j}[\Phi] }{{\rm Tr}  (C_[\Phi]) } -  \sum_{k=1}^r \frac{ \displaystyle 2 \lambda_k[\Phi]}{{\rm Tr}  (C_[\Phi])^2 } \\
    &= \sum_{j=1 \atop j\neq i}^n \sum_{k=1}^r \frac{ 2\zeta_{k,i}[\Phi]\zeta_{k,j}[\Phi] }{{\rm Tr}  (C_[\Phi]) } + \sum_{k=1}^r \Big(\frac{2{\zeta_{k,i}[\Phi]}^2}{{\rm Tr} (C_[\Phi])} - \frac{ \displaystyle 2 \lambda_k[\Phi]}{{\rm Tr}  (C_[\Phi])^2 } \Big)\\ 
    &=\sum_{j=1}^n \sum_{k=1}^r \left( \frac{ 2\zeta_{k,i}[\Phi]\zeta_{k,j}[\Phi] }{{\rm Tr} (C[\Phi]) } -\frac{ \displaystyle 2\lambda_k[\Phi]}{{\rm Tr} (C[\Phi])^2 }\delta_{ij}\right)
\end{align*}
with $\delta_{ij}$ is the Kronecker delta. Finlay, we obtain 
    \begin{align}
    D J_r[\Phi][\Psi] &=\sum_{i=1}^n\sum_{j=1}^n \sum_{k=1}^r \left( \frac{ 2\zeta_{k,i}[\Phi]\zeta_{k,j}[\Phi] }{{\rm Tr} (C[\Phi]) } -\frac{ \displaystyle 2\lambda_k[\Phi]}{{\rm Tr} (C[\Phi])^2 }\delta_{ij}\right)\mathbf{Z}_{ji}[\Phi][\Psi]  .
\end{align}

\end{proof}
We introduce the \Ale{following} shorthand notation:
\begin{align}\label{f_i expression}
\f_i [\Phi] :=\sum_{j=1}^n \sum_{k=1}^r \left( \frac{ 2\zeta_{k,i}[\Phi]\zeta_{k,j}[\Phi] }{{\rm Tr} (C[\Phi]) } -\frac{ \displaystyle 2\lambda_k[\Phi]}{{\rm Tr} (C[\Phi])^2 }\delta_{ij}\right) (u_j \circ \bphi_j)(\boldsymbol{x})  (\Vec{\nabla} u_i\circ\bphi_i)(\boldsymbol{x}),    
\end{align}
and $\F[\Phi]:=(\f_1[\Phi], \cdots , \f_n[\Phi])$. Then, Lemma \ref{lem : diff J} can be written as 

\begin{align} \label{eq:J_diff}
D J_r[\Phi][\Psi] = \sum_i^n \langle f_i[\Phi],\bpsi_i\rangle_{\Omega_0}.    
\end{align}

\subsubsection*{Riesz representation}

\Ale{The next step is to transform the differential of $J_r$ into a gradient by choosing a suitable inner product, say $\big\langle\cdot,\cdot\big\rangle_\mathcal{H}$, where $\mathcal{H}$ is some embedding Hilbert space. The gradient of $J_r$ is then the Riesz representative of $DJ_r$ using the chosen inner product, i.e.~we have }

\begin{align}
     \big\langle \nabla J_r[\Phi],\,\Psi\big\rangle_\mathcal{H} \;=\; D J_r[\Phi][\Psi]\,.
\end{align}




Once the gradient $\nabla J_r[\Phi]$ is available, \Ale{a minimizer of the functional $J_r$ can be sought by a} steepest ascent algorithm. Starting from an initial guess $\Phi^{(0)} = (\bphi_i^{(0)})_{1\leq i \leq n} \in \mathbf{M}$ (see Remark \ref{Initialization}), we update the morphings in the ascent direction. Specifically, at iteration $m>0$, we compute the gradient $\nabla J_r[\Phi^{(m)}]:= (\bu_i^{(m)})_{1\leq i\leq n}  \subset \mathcal{H}^{(m)} $, as the Riesz representative of the differential at $\Phi^{(m)}$, and then update each morphing \Ale{as}
\begin{align} \label{gradient update}
\bphi_i^{(m+1)} \;=\; \bphi_i^{(m)} \;+\; \epsilon^{(m)}\, \bu_i^{(m)}\,, \quad \forall \, 1\leq i \leq n.     
\end{align}
Here, $\epsilon^{(m)}>0$ is a step-size parameter taken sufficiently small to maintain stability of the iterative process and admissibility \Ale{(bijectivity)} of the updated morphings $\bphi_i^{(m+1)}$. Because $J_r$ is non-convex, the steepest ascent method does not guarantee a global \Ale{minimizer}. However, it provides a tractable approach to iteratively improve \Ale{the value of} $J_r$ and is used as the backbone of our algorithm.
\begin{rem}[Initialization]\label{Initialization}
    A crucial prerequisite for \Ale{the above} algorithm is \Ale{an} initial morphing family $\Phi^{(0)}$ that lies in $\mathbf{M}$, the set of admissible (bijective) maps. That is, we require an initial guess consisting of morphings $\bphi_i^{(0)}$ such that $\bphi_i^{(0)}(\Omega_0)=\Omega_i$, for all $1\leq i \leq n$. Computing such morphings between arbitrary geometries is itself a non-trivial task that can be accomplished by existing morphing or registration techniques. In \Ale{what follows}, we use either RBF morphing \cite{de2007mesh} or elasticity-based morphing \cite{kabalan2025elasticity}. RBF morphing can suffer from difficulties in the case of non-parametrized geometries and when the \Ale{knowledge of the image} of the control points (usually the boundary) is not easily available. These difficulties can be tackled using the elasticity-based morphing methodology developed in \cite{kabalan2025elasticity}, which tends to produce smooth and invertible mappings even for complex, non-parameterized geometries. 
\end{rem}

\subsection{Preconditioning}\label{sec:inner product}

The choice of the inner product $\langle\cdot,\cdot\rangle_{\mathcal{H}}$ is a crucial design decision that has significant theoretical and practical implications. Recall that the gradient is defined via the Riesz representation with respect to this inner product. Different inner products yield different gradient vectors (even though they represent the same underlying differential), implying that the steepest-ascent direction depends on how  distances and angles are measured in the functional space. In essence, the inner product acts as a form of preconditioning for the gradient ascent procedure. There are two primary requirements for the inner product in the present morphing optimization. 
\begin{enumerate}
    \item Regularization: the inner product should be chosen so that the resulting gradient field $ \nabla J_r[\Phi]$ is sufficiently smooth and free of high-frequency oscillations or non-physical features. \Ale{For instance, with the naive choice} $\mathcal{H} := [L^2(\Omega_0)]^n$ with the $L^2$-inner product, the gradient becomes essentially the raw sensitivity field $(\f_i[\Phi])_{1\leq i\leq n}$ from \eqref{eq:J_diff}. \Ale{In our} numerical experiments, \Ale{we observed that} the fields $\f_i[\Phi]$ can be extremely irregular and concentrated in localized regions. This is \Ale{somehow expected} since $J_r$ depends on the \Ale{products between} the solution fields. A noisy or non-smooth gradient direction can cause the morphing update \eqref{gradient update} to distort the mesh or even violate invertibility. Therefore, it is critical to define an inner product that includes derivative terms (making it closer to an $H^1$ inner product). This has the effect of filtering out the oscillatory components of $\f_i$ and yielding smoother displacement fields $\bu^{(m)}_i$ at each iteration. 
    \item Conservation of mapping constraints: the inner product must be compatible with the geometric constraints related to our search space $\mathbf{M}$. \Ale{We require} that $\bphi_i^{(m)}(\partial\Omega_0) = \partial\Omega_i$, \Ale{which, in return, requires that} the update displacement field $\bu_i^{(m)}$ (the Riesz representative) be tangential to the boundary $\partial\Omega_i$. \Ale{Specifically}, if $\bn_{\bphi_i}(\boldsymbol{y})$ denotes the unit outward normal at a boundary point $\boldsymbol{y} \in \partial\Omega_i$, then we require $\bu_i(\boldsymbol{x})\cdot \bn_{\bphi_i}(\boldsymbol{y})=0$, \Ale{for all $\boldsymbol{x}\in\partial\Omega_0$ with $\boldsymbol{y}:=\bphi_i(\boldsymbol{x})$}. This tangential update condition ensures that $\bphi_i^{(m+1)}(\boldsymbol{x}) = \bphi_i^{(m)}(\boldsymbol{x}) + \epsilon ^{(m)}\bu_i^{(m)}(\boldsymbol{x})$ still maps $\boldsymbol{x}$ onto (approximately) $\partial\Omega_i$. 
\end{enumerate}

We work with each component of $\nabla J_r[\Phi^{(m)}]= (\bu_i^{(m)})_{1\leq i\leq n} \subset \mathcal{H}^{(m)}$ separately. We define a linear elasticity-based inner product. Concretely, let $\bu,\bv \in \boldsymbol{H}^1(\Omega_0)$ \Ale{be two displacement fields on $\Omega_0$}. We define
 $$
(\bu,\bv) \mapsto a(\bu,\bv) := \int_{ \Omega_0} \sigma ( \bu):\varepsilon (\bv) d\boldsymbol{x},
$$
where $\varepsilon(\bu)$ and $\sigma(\bu)$ are respectively the linearized strain and stress tensors defined by
\begin{align*}
    &\varepsilon(\bu) := \frac{1}{2} (\nabla \bu + \nabla \bu^T),\\
    &\sigma(\bu) :=  \frac{E}{(1+\nu)}\varepsilon(\bu) + \frac{E\nu}{(1+\nu)(1-\nu)} {\rm {\rm Tr}}(\varepsilon(\bu))I.
\end{align*}
Here, $E>0$ and $-1<\nu<\frac{1}{2}$ are, respectively, the Young modulus and the Poisson ratio.

To enforce the tangential boundary condition, we restrict the displacement space to $H^1_{\bn_{\bphi}}(\Omega_0) := \{\,\bu \in H^1(\Omega_0) : \bu \cdot \bn_{\bphi}\circ\bphi = 0 \}$, \Ale{so that} the admissible displacements $\bu$ are those that produce tangential displacements of the boundary \Ale{of $\Omega_0$}. This subspace $H^1_{\bn_{\bphi}}$ can be thought of as the tangent space to the manifold $\mathbf{M}$ at the current morphing $\bphi$. Variations in this tangent space will deform the interior of the domain while sliding along the boundary, thus \Ale{preserving  the target} boundary (up to higher order).

In practice, we impose the condition $\bu \cdot \bn_{\bphi} \circ \bphi = 0$ using a penalty method. We modify the bilinear form $a(\bu,\bv)$ by adding a penalty term on the boundary normal component:
\begin{align} \label{eq: inner product normal}
 \forall \bu,\bv \in \boldsymbol H^1(\Omega_0),\quad a(\bu,\bv) := \int_{ \Omega_0} \sigma ( \bu):\varepsilon (\bv) d\boldsymbol{x} + \alpha \int_{ \bphi(\partial\Omega_0)} (\bu \circ \bphi^{-1}){\cdot} \bn_{\bphi} (\bv\circ \bphi^{-1}){\cdot} \bn_{\bphi}  ds .
\end{align}
This bilinear form defines an inner product in $ H^1(\Omega_0)$ \cite{kabalan2025elasticity}. The scalar user-dependent parameter $\alpha \gg 1$ is a large weight (for example, we used $\alpha := 10^{12}$ in our \Ale{numerical} experiments). Altogether, we proceed as follows: For all $1\leq i \leq n$,
\begin{enumerate}
    \item Compute the \Ale{sensitivity term} $\f_i[\Phi^{(m)}]$.
    \item Compute the unique element $\bu_i^{(m)}\in H_{\bn^{(m)}}$, with $\bn^{(m)}:=\bn_{\bphi_i^{(m)}}$ , such that \begin{align}
        a(\bu_i^{(m)},\bv)= \langle \f_i[\Phi^{(m)}] , \bv \rangle_{\Omega_0} \quad \forall \bv \in H_{\bn^{(m)}}.
    \end{align} 
    \item Update the morphing using the gradient \Ale{ascent} scheme as $$
    \bphi_i^{(m+1)}=\bphi_i^{(m)}+ \epsilon \bu_i^{(m)}.
    $$
\end{enumerate}
\begin{rem}[Safeguard check]
    If the geometric error between $\partial\bphi_i^{(m)}(\Omega_0)$ and $\partial\Omega_i$ \Ale{is deemed too} large, one can perform an \Ale{additional} correction step by projecting $\partial\bphi^{(m)}_i(\Omega_0)$ onto  $\partial\Omega_i$. However, this happened very rarely in our numerical experiments.
\end{rem}

\subsection{Bijectivity}\label{sec: bijectivity}
A distinguishing challenge in the \Ale{above minimization} problem \eqref{optimization statement} is the requirement that \Ale{the morphings} $\bphi_i^{(m)}:\Omega_0 \to \Omega_i$ remain bijective at all iterations. This constraint is essential to preserve the physical correspondence between points in the reference and target domains and to guarantee that the pulled-back solution fields $u^{\rm ref}_i$ defined in \eqref{eq: u ref} remain meaningful. However, enforcing global bijectivity is highly nontrivial: the set of all diffeomorphisms is an open, non-convex set, and naive gradient updates such as \eqref{gradient update} can easily drive the updated morphing family $\Phi^{(m)}$ out of this set.

\subsubsection*{Penalized objective functional}

We incorporate a penalty strategy that introduces a barrier to non-bijective transformations. Instead of maximizing $J_r[\Phi]$ alone, we include a penalty functional $E[\bphi_i]$ on each morphing $\bphi_i$ which becomes large (ideally unbounded) as $\bphi_i$ looses invertibility. Specifically, we consider the augmented objective functional
\begin{align} \label{eq: I definition}
    I_{r}[\Phi]:= J_{r}[\Phi] - c_1 \sum_{i=1}^n E[\bphi_i],
\end{align}
and seek  \begin{align} \label{eq: optimization I}
    \Phi^* = \arg\max_{\Phi\in \mathbf{M}} I_r[\Phi].
\end{align} 
Here $c_1>0$ is a penalty parameter that tunes the trade-off between the original objective and the invertibility enforcement. Moreover, we choose $E$ such that it tends to $+\infty$ if $\bphi_i$ looses invertibility. In \Ale{this situation} $I_r[\Phi]$ tends to $-\infty$, making any non-invertible component of the morphing family $\Phi$ an undesirable \Ale{candidate}. Thus, by maximizing $I_r$ instead of $J_r$, we bias the search toward the subset of $\mathbf{M}$ consisting of well-behaved deformations.

\subsubsection*{Penalty functional design}
We propose two choices for $E[\bphi]$ inspired by continuum mechanics, which penalize distortion and singular Jacobian matrix.
\begin{enumerate}
    \item Linear elastic energy: we define $E[\bphi]$ as the elastic strain energy of the deformation $\bphi$ relative to the identity map. Assuming small strains, we \Ale{set}
\begin{align} \label{linear elastic energy}
     E_{\text{lin}}[\bphi] := \frac{1}{2}\int_{ \Omega_0} \sigma ( \bphi-\bId):\varepsilon (\bphi-\bId) d\boldsymbol{x},
\end{align}
  Here, $\bId$ is the identity mapping on $\Omega_0$, $\mathbf{u}:=\bphi - \bId$ is the displacement field, $\varepsilon(\mathbf{u}) := \tfrac{1}{2}(\nabla \mathbf{u} + \nabla \mathbf{u}^{\rm T})$ is the linearized-strain tensor, and $\sigma(\mathbf{u})$ is the corresponding stress tensor. This choice of $E_{\text{lin}}$ penalizes deviations from the identity map, with larger penalties for large strains or shear/distortion. It tends to keep $\bphi$ in the regime of mild, smooth deformations and prevents excessive element distortion that could lead to Jacobian sign change. However, $E_{\text{lin}}$ alone does not in general diverge fast enough as $\det(\nabla\bphi)\to 0$. 
  \item Nonlinear (Neo-Hookean) elastic energy: for a stronger guarantee against loss of injectivity, we use a hyperelastic energy functional. Specifically, we consider a compressible Neo-Hookean model given by

 \begin{align} \label{non linear elastic energy}
     E_{\rm NH}[\bphi] := \int_{ \Omega_0} \frac{\mu}{2} (\text{trace}(F[\Phi]^{\rm T}F[\Phi])-3)+\lambda(J[\Phi]^2-1)-\left|\frac{\lambda}{2}+\mu\right|\text{ln}J[\Phi] d\boldsymbol{x}.
\end{align}
  Here $F[\Phi] := \nabla \bphi$ is the deformation gradient, $J[\Phi] := \det(F[\Phi])$ is the Jacobian determinant of $\bphi$, and $\mu,\lambda>0$ are material parameters (Lamé constants). The energy functional $E_{\rm NH}$ grows rapidly if the morphing $\bphi$ undergoes large stretching or compression. Crucially, as $J \to 0$ (volume collapses, indicating a loss of injectivity), the term $- \left|\frac{\lambda}{2} + \mu\right|\ln J[\Phi] $ in $E_{\text{NH}}$ blows up to $+\infty$.  Thus, $E_{\text{NH}}$ provides an infinite-energy barrier to any morphing that approaches a fold-over or degenerate Jacobian. Neo-Hookean and related hyperelastic energies are standard in enforcing diffeomorphic transformations in topology optimization and image registration contexts, due to their ability to prevent element inversion.
\end{enumerate}
For both  choices, we use the same gradient \Ale{ascent} algorithm with the inner product defined in \eqref{eq: inner product normal}. We note that adding such penalties does not guarantee strict bijectivity for a given finite $c_1$, but by taking $c_1$ sufficiently large, the gradient \Ale{ascent algorithm} is expected to prefer configurations far from the non-invertible regime. The value of the parameter $c_1$ can be \Ale{progressively decreased} through the iterations as we now explain.

\subsubsection{Continuation for $c_1$}
\begin{enumerate}
    \item We solve the maximization problem \eqref{eq: optimization I} using the gradient-ascent algorithm with $c_1:=c_1^0$, where $c_1^0$ should be chosen sufficiently large in order to ensure bijectivity. This produces the morphing family $\Phi^0$.
    \item For all $k\geq 1$, we set $\displaystyle c_1^k:=\frac{1}{2}c_1^{k-1}$, and solve the maximization for $\Phi^k$ by initializing with $\Phi^{(0),k}:=\Phi^{k-1}$.
    \item We iterate over $k$ as long as there are no elements of the morphed meshes that are inverted.
\end{enumerate}
The idea is to decrease progressively the weighting coefficient $c_1$, since ideally we want to solve \eqref{eq: I definition} for $c_1=0$.
This leads to an outer iterative procedure instead by $k$ embedding the gradient ascent algorithm for the functional $I_r$ with varying coefficient $c_1$. We notice that full convergence for the intermediate morphings $\Phi^k$ is generally not necessary. We call the resulting two-loop procedure a continuation method on $c_1$.

\subsection{Computational cost and acceleration strategies}

Solving the optimization problem \eqref{eq: optimization I} requires \AK{four} \Ale{main tasks at each iteration:} (i) evaluating, \AK{for all $1 \leq i \leq n$}, $\f_i[\Phi]$ (thus $u_i$ and its gradient) at the integration points of the deformed reference mesh, (ii) computing the energy $E_{\rm lin}$ or $E_{\rm NH}$, (iii) assembling the stiffness matrix associated with the bilinear form $a$, (iv) solving the linear systems to compute the Riesz representative. Potential computational bottlenecks come from having a very fine reference mesh, which increases the degrees of freedom and the number of integration points. We use an acceleration strategy consisting of using a coarse reference mesh to compute the optimal morphings, and then evaluating the morphings on the fine mesh using interpolation. This greatly decreases optimization time while providing accurate results. To further accelerate the computations, when needed, other techniques can be used like reduced quadrature formulae or randomized linear algebra. However, this was not needed in the examples considered below.

\section{Preliminary examples}
To showcase the different components of the methodology introduced in the previous section, we present here some preliminary examples. First, we start \Ale{with} some \Ale{simplifying} results \Ale{in specific situations}.

\subsection{Polytopal domains and linear elasticity-based energy}
In the specific case where $\Omega_0=\Omega_i$ for all $1\leq i \leq n$, and $\Omega_0$ is a polytopal domain (polygon in 2D and polyhedron in 3D), the enforcement of the boundary condition on the morphing can be simplified. In this case, the outward normal direction $n_{\bphi}$ on each facet is constant, independent of the point on that facet and, thus, independent of the deformation as long as the facet in question always maps onto itself. Thus, the space $H^1_{\bn_{\bphi}}(\Omega_0)$ does not change with $\bphi$ during iterations, and the \Ale{bilinear form} product $a$ in \eqref{eq: inner product normal} can be assembled once at the start and reused for all iterations. To this end, we replace the bilinear form $a$ by
\begin{align} \label{eq: inner product normal fixed}
 \forall \bu,\bv \in \boldsymbol H^1(\Omega_0),\quad a(\bu,\bv) := \int_{ \Omega_0} \sigma ( \bu):\varepsilon (\bv) d\boldsymbol{x} + \alpha \int_{ \partial\Omega_0} \bu \cdot \bn \bv\cdot \bn  ds .
\end{align}
This avoids the need to update the stiffness matrix at each iteration, yielding a significant computational speedup in the offline phase. In other words, for polytopal domains, the \Ale{bilinear form} is fixed, whereas for general curved domains, the \Ale{bilinear form} is state-dependent, requiring re-computation and re-assembly at each step.
\Ale{Another simplification is achieved in the case of linear-elasticity based penalty in \eqref{linear elastic energy}}. In this situation, we can further simply the gradient \Ale{algorithm} by avoiding the computation of the differential of the energy term $E_{\rm lin}$. In fact, we can suppose that $E_{\text{lin}} [\bphi]= \frac{1}{2}a(\bphi-\bId,\bphi-\bId)$ in this case, \Ale{so that} we obtain 
\begin{align*}
    DI_{r}[\Phi][\Psi]&=DJ_r[\Phi][\Psi]-c_1\sum_{i=1}^n DE_{\text{lin}}[\bphi_i][\bpsi_i] \\
    &=\sum_{i=1}^n \langle \f_i[\Phi], \bpsi_i \rangle -c_1 \sum_{i=1}^n a(\bphi_i-\bId,\psi_i).
\end{align*}

\Ale{Therefore, for all $1\leq i \leq n$, the $i$-th component of the Riesz representative} of $DI_r[\Phi]$, denoted as $\Bar{\bu}_i[\Phi]$, \AK{is obtained from the $i$-th component of the Riesz representative of $DJ_r[\Phi]$, denoted as $\bu_i[\Phi]$, by a simple substraction, i.e.~we have }  
$$\Bar{\bu}_i[\Phi]=\bu_i[\Phi]-c_1(\bphi_i- \bId).$$ Hence the gradient scheme \eqref{gradient update} becomes
\begin{align}\label{gradient update I elasticity}
\bphi^{(m+1)}_i=\bphi^{(m)}_i + \epsilon^{(m)} \left(\bu_i[\Phi^{(m)}]- c_1 \left(\bphi_i^{(m)}-\bId\right)\right), \quad \forall \, \Ale{1\leq i \leq n.} 
\end{align}
We emphasize that the above equation is only valid for polytopal domains and linear elasticity-based penalty.


\subsection{Toy example}
We consider the domain $\Omega_0:=(-1,1)~\times~(-1.25,1.25)$, and define, for all $1\leq i \leq n:=30$, the functions $u_i$ on $\Omega_i := \Omega_0$ as 
\begin{align}
    u_i(x,y):= \exp\left(-\frac{(\beta_i(x+1)-y)^2}{0.05}\right),
\end{align}
with $\beta_i $ uniformly sampled in $[-0.38,0.38]$. In Figure \ref{fig: 3 samples}, we show $u_i$ for three values of $\beta_i$. We use the gradient algorithm to solve the \Ale{gradient ascent} problem with $\epsilon := 2.5$ . We fix the Young modulus $E:=1$, and the Poisson ratio $\nu:=0.3$ for the elasticity biliear form $a$, and we choose \Ale{just} $r:=1$ mode.
\begin{figure}[h] 
     \centering
     \begin{subfigure}[b]{0.3\textwidth}
         \centering
        \includegraphics[scale=0.2]{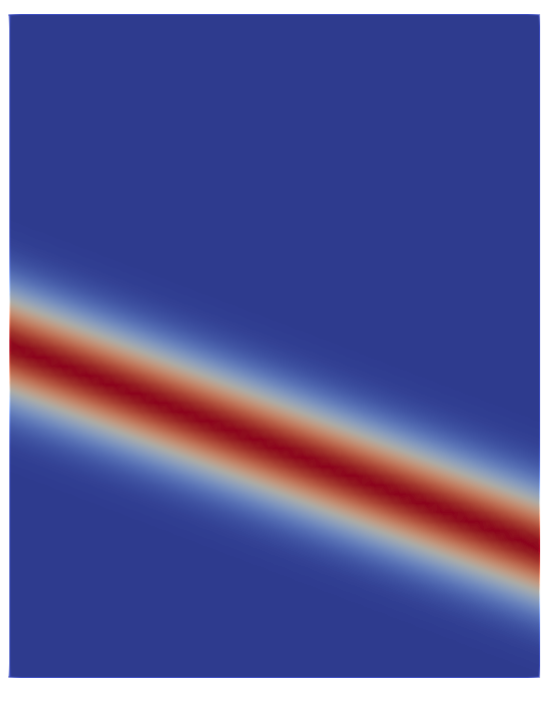}
     \caption{$\beta=-0.38$}
     \end{subfigure}
     \hfill
     \begin{subfigure}[b]{0.3\textwidth}
         \centering
        \includegraphics[scale=0.2]{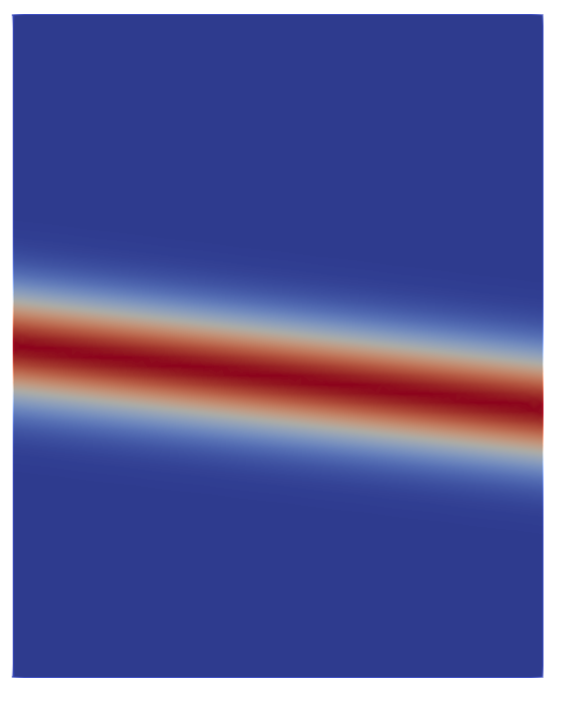}
     \caption{ $\beta=-0.126$}
     \end{subfigure}
          \hfill
     \begin{subfigure}[b]{0.3\textwidth}
         \centering
        \includegraphics[scale=0.32]{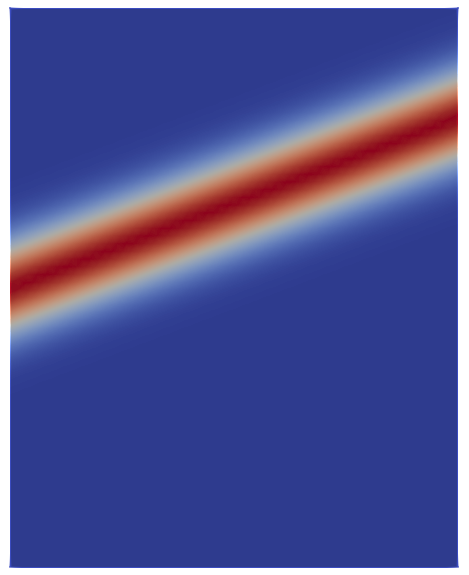}
     \caption{ $\beta=0.38$ }
     \end{subfigure}
\caption{\label{fig: 3 samples} $u_i$ for three different values of $\beta$.}
\end{figure}
\subsubsection{Linear elasticity-based energy}
First, we run the algorithm with $c_1:=5\times 10^{-3}$ using the gradient ascent method with \Ale{the} linear elasticity-based penalty energy $E_{\text{lin}}$ \Ale{defined in \eqref{linear elastic energy}}. We also test the effect of performing numerical continuation on $c_1$, as discussed \Ale{above}. In this case, we set $c_1^0:=1$, and $c_1^k$ is modified when $\displaystyle \left| \frac{J[\Phi^{(m),k}]-J[\Phi^{(m-1),k}]}{J[\Phi^{(m),k}]} \right| < 10^{-4}$. In \Ale{our numerical results}, $c_1^k$ converges to zero and reaches values below $10^{-8}$ in about 200 iterations. In Figures \ref{fig: 3 samples morphed u} and \ref{fig: 3 samples morphed u continuation}, we show $u_i\circ \bphi_i$ for the same three different values of $\beta$ as in Figure \ref{fig: 3 samples}, with and without continuation \Ale{on $c_1$}. For the case with continuation, the samples are perfectly aligned and can be reduced to one mode. This is not the case when fixing $c_1$ \Ale{to} $5\times 10^{-3}$. This is due to the fact that some of the 30 computed morphings are actually not bijective. In Figures \ref{fig: 3 morphings} and \ref{fig: 3 morphings continuation}, we show the morphed meshes for the three values of $\beta$. We can clearly see the flipped elements in Figure \ref{fig: 3 morphings}.

\begin{figure}[h] 
     \centering
     \begin{subfigure}[b]{0.3\textwidth}
         \centering
        \includegraphics[scale=0.3]{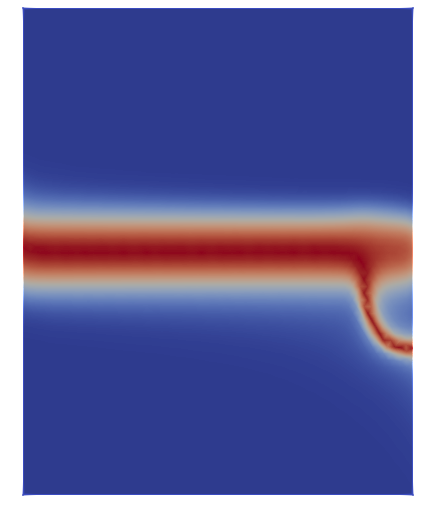}
     \caption{$\beta=-0.38$}
     \end{subfigure}
     \hfill
     \begin{subfigure}[b]{0.3\textwidth}
         \centering
        \includegraphics[scale=0.3]{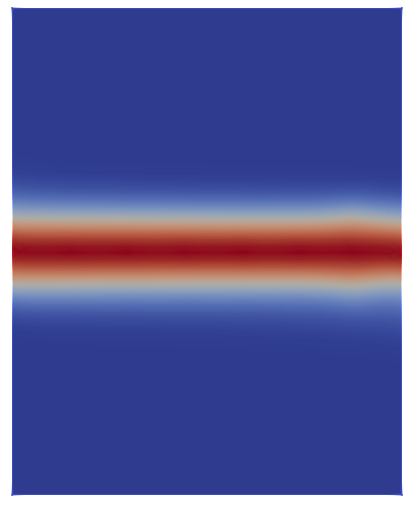}
     \caption{ $\beta=-0.126$}
     \end{subfigure}
          \hfill
     \begin{subfigure}[b]{0.3\textwidth}
         \centering
        \includegraphics[scale=0.3]{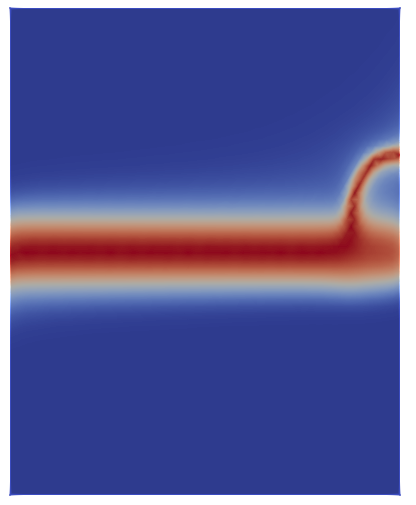}
     \caption{ $\beta=0.38$ }
     \end{subfigure}
\caption{\label{fig: 3 samples morphed u} $u_i \circ \bphi_i$ for three different values of $\beta$ without using the continuation on $ c_1$.}
\end{figure}

\begin{figure}[h!] 
     \centering
     \begin{subfigure}[b]{0.3\textwidth}
         \centering
        \includegraphics[scale=0.29]{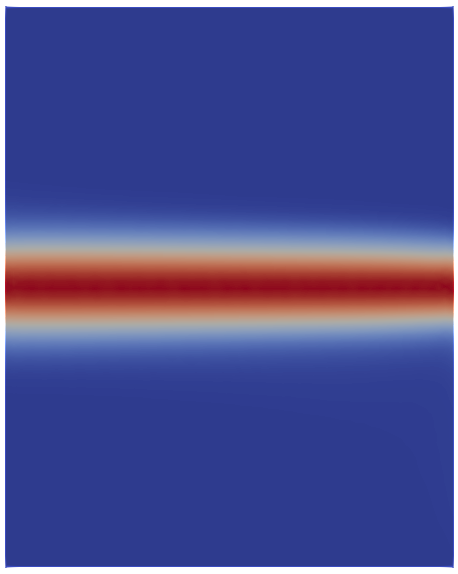}
     \caption{$\beta=-0.38$}
     \end{subfigure}
     \hfill
     \begin{subfigure}[b]{0.3\textwidth}
         \centering
        \includegraphics[scale=0.29]{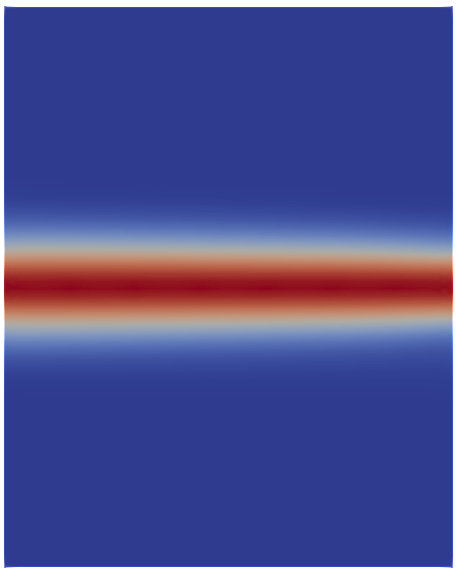}
     \caption{ $\beta=-0.126$}
     \end{subfigure}
          \hfill
     \begin{subfigure}[b]{0.3\textwidth}
         \centering
        \includegraphics[scale=0.29]{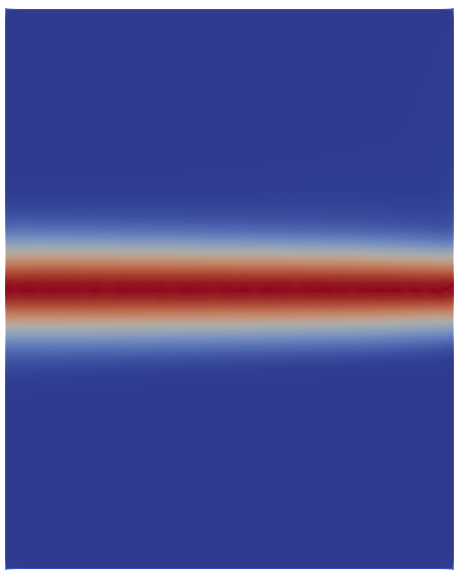}
     \caption{ $\beta=0.38$ }
     \end{subfigure}
\caption{\label{fig: 3 samples morphed u continuation} $u_i\circ \bphi_i$ for three different values of $\beta$, using continuation \Ale{on $c_1$}.}
\end{figure}

\begin{figure}[h!] 
     \centering
     \begin{subfigure}[b]{0.3\textwidth}
         \centering
        \includegraphics[scale=0.345]{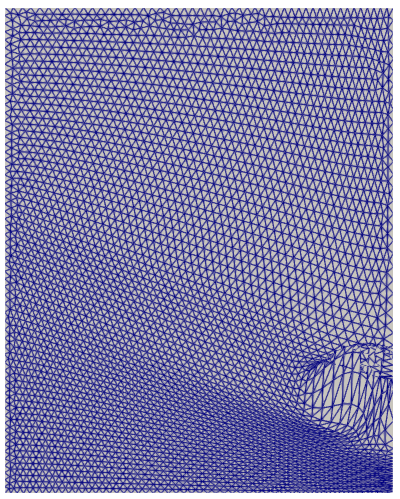}
     \caption{$\beta=-0.38$}
     \end{subfigure}
     \hfill
     \begin{subfigure}[b]{0.3\textwidth}
         \centering
        \includegraphics[scale=0.345]{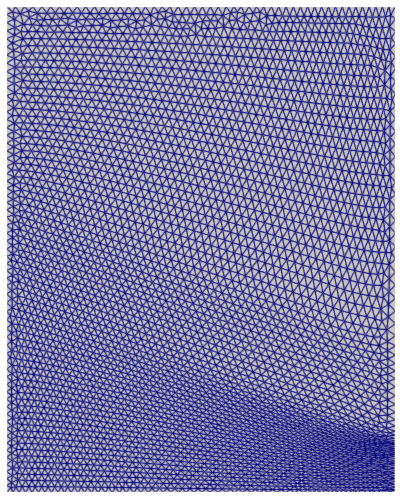}
     \caption{ $\beta=- 0.354$}
     \end{subfigure}
          \hfill
     \begin{subfigure}[b]{0.3\textwidth}
         \centering
        \includegraphics[scale=0.3]{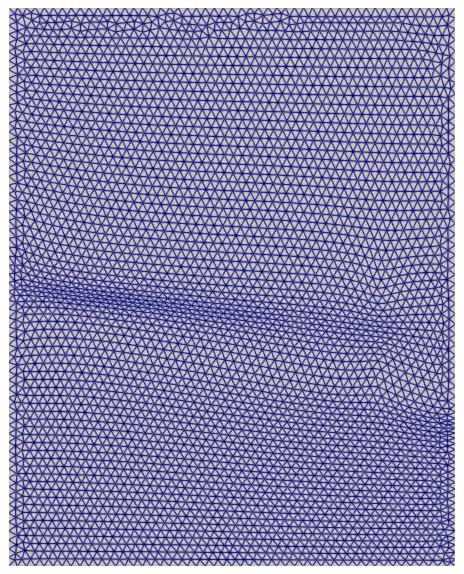}
     \caption{ $\beta= -0.126$ }
     \end{subfigure}
\caption{\label{fig: 3 morphings} $\bphi_i$ for three different values of $\beta$ without using continuation on $c_1$.}
\end{figure}

\begin{figure}[h!] 
     \centering
     \begin{subfigure}[b]{0.3\textwidth}
         \centering
        \includegraphics[scale=0.3]{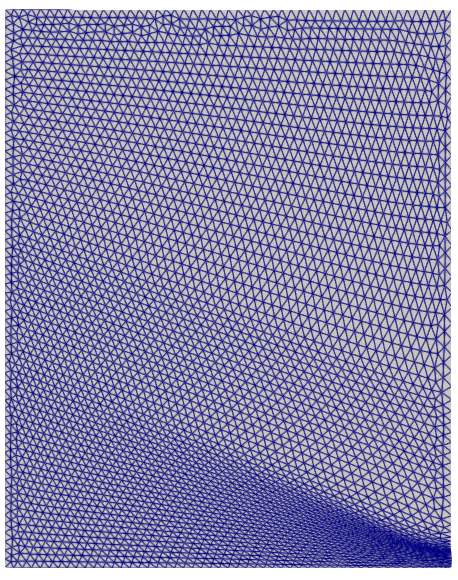}
     \caption{$\beta=-0.38$}
     \end{subfigure}
     \hfill
     \begin{subfigure}[b]{0.3\textwidth}
         \centering
        \includegraphics[scale=0.3]{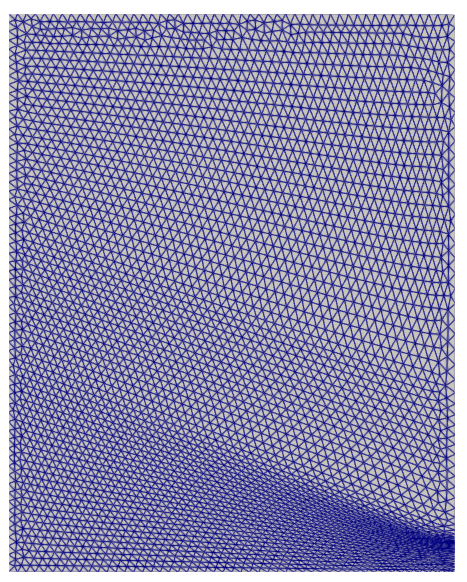}
     \caption{ $\beta=- 0.354$}
     \end{subfigure}
          \hfill
     \begin{subfigure}[b]{0.3\textwidth}
         \centering
        \includegraphics[scale=0.3]{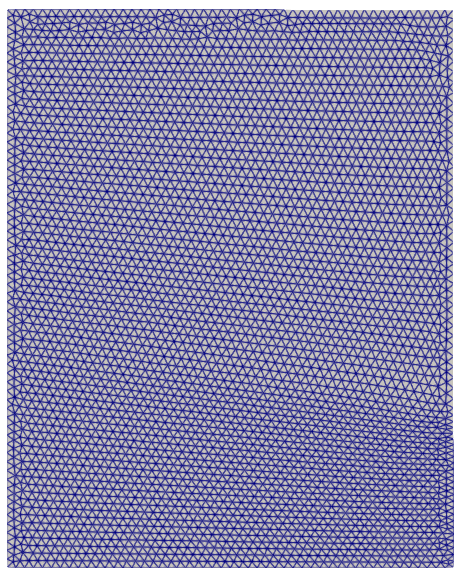}
     \caption{ $\beta= -0.126$ }
     \end{subfigure}
\caption{\label{fig: 3 morphings continuation} $\bphi_i$ for three different values of $\beta$ computed using the continuation method.}
\end{figure}

Next, we plot in Figure \ref{fig: J_phi_comparaison} the evolution of $1-J_r[\Phi]$ during the optimization process for both cases. At the begining of the algorithm, the value of $1-J_r[\Phi]$ decreases more rapidly without continuation than with \Ale{continuation}. The sharp changes in the curve with continuation are due to changes in $c_1^k$. Using continuation, the algorithm reaches $10^{-4}$ in about 400 iterations, an order of magnitude smaller than the \Ale{value reached without continuation in the same number of iterations}.

\begin{figure}[ht] 
\begin{center}
\includegraphics[scale=0.5]{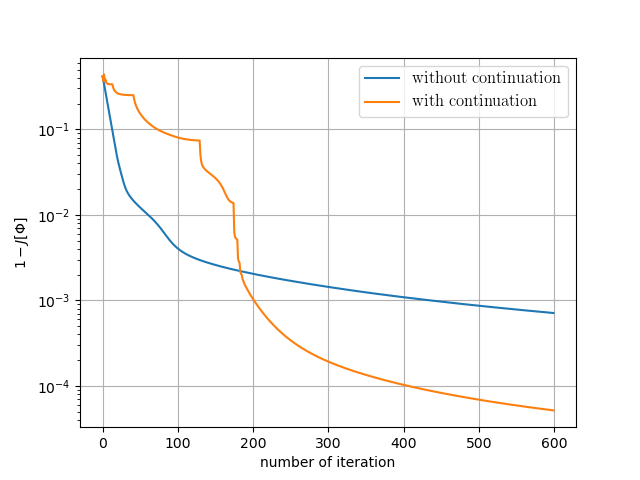}
\caption{ \label{fig: J_phi_comparaison} Evolution of $1-J_r[\Phi]$ (in logarithmic scale) with respect to the number of iterations.}
\end{center}
\end{figure}

\subsubsection{Non linear elasticity-based penalty}
We now test the non linear elasticity-based penalty energy $E_{\text{NH}}$ defined in \eqref{non linear elastic energy}. We fix $\mu:=1$, $\lambda:=0.1$ and $c_1:= 5 \times 10^{-3}$ as for the linear case above. In this setting, adding the nonlinear energy term keeps all the morphings bijective. The functional $J_r[\Phi]$ reaches the value $0.995$.  In Figure \ref{fig: 3 samples morphed u NH}, we plot the morphed fields $u_i\circ \bphi_i$ for different values of the parameter $\beta$. We can notice some slight variations in contrast to the previous case. This is expected since we keep here the penalty term.

\begin{figure}[ht] 
     \centering
     \begin{subfigure}[b]{0.3\textwidth}
         \centering
        \includegraphics[scale=0.35]{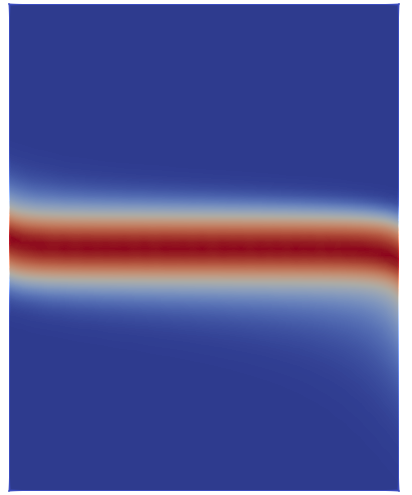}
     \caption{$\beta=-0.38$}
     \end{subfigure}
     \hfill
     \begin{subfigure}[b]{0.3\textwidth}
         \centering
        \includegraphics[scale=0.35]{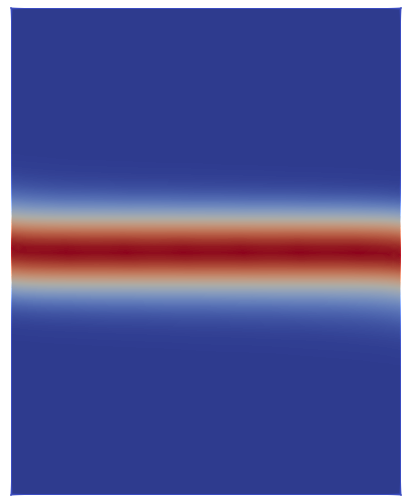}
     \caption{ $\beta=-0.126$}
     \end{subfigure}
          \hfill
     \begin{subfigure}[b]{0.3\textwidth}
         \centering
        \includegraphics[scale=0.35]{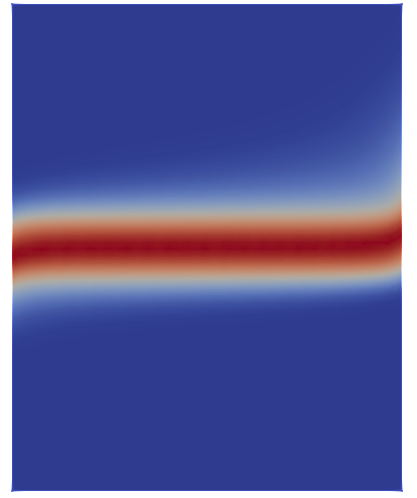}
     \caption{ $\beta=0.38$ }
     \end{subfigure}
\caption{\label{fig: 3 samples morphed u NH} $u_i\circ \bphi_i$ for three different values of $\beta$, using the non linear Neo-Hookean elasticity-based penalty.}
\end{figure}

\subsubsection{Varying $r$}
Finally, we test the algorithm by setting $r:=2$. In this case, we maximize the energy in the first two modes instead of one mode as in the previous setting. In Figure \ref{fig: 4 morphings 2modes}, we plot the morphed fields for \Ale{four} values of the parameter $\beta$. We can see that the \Ale{fields} align on two configurations automatically, without any \textit{a priori} knowledge of these two configurations. In Figure  \ref{fig: POD modes}, we plot the first two POD modes of $\{u_i\circ\bphi_i\}_{1\leq i \leq 30}$.

\begin{figure}[H] 
     \centering
     \begin{subfigure}[b]{0.24\textwidth}
         \centering
        \includegraphics[scale=0.27]{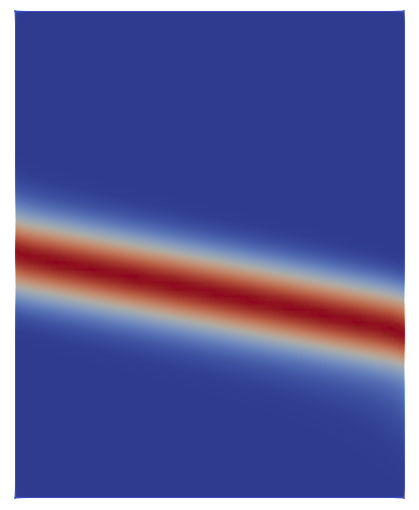}
     \caption{$\beta=-0.38$}
     \end{subfigure}
     \hfill 
     \begin{subfigure}[b]{0.24\textwidth}
         \centering
        \includegraphics[scale=0.27]{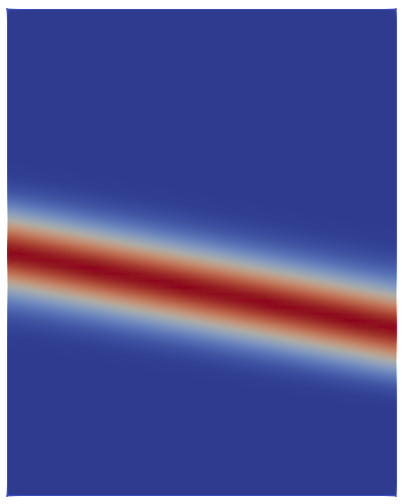}
     \caption{ $\beta=- 0.126$}
     \end{subfigure}
          \hfill
     \begin{subfigure}[b]{0.24\textwidth}
         \centering
        \includegraphics[scale=0.27]{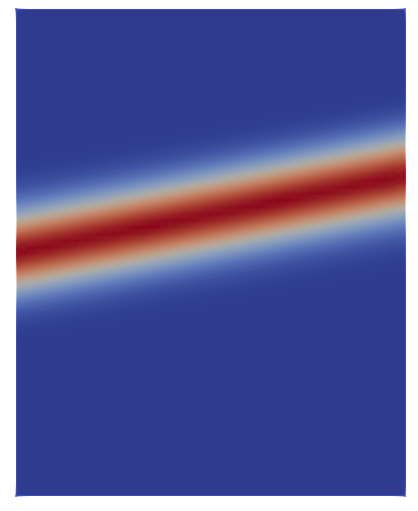}
     \caption{ $\beta= 0.126$ }
     \end{subfigure}
              \hfill
     \begin{subfigure}[b]{0.24\textwidth}
         \centering
        \includegraphics[scale=0.27]{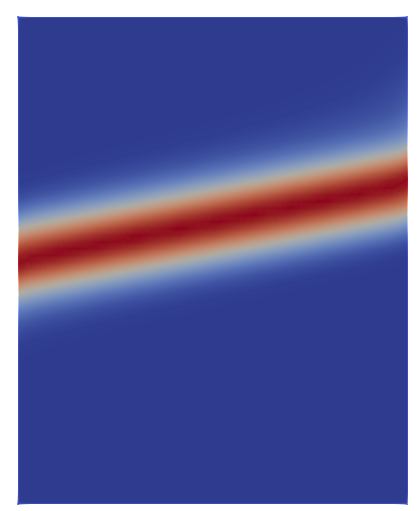}
     \caption{ $\beta= 0.38$ }
     \end{subfigure}

\caption{\label{fig: 4 morphings 2modes} $u_i\circ\bphi_i$ for four values of $\beta_i$, all computed for $r=2$.}
\end{figure}

\begin{figure}[ht] 
     \centering
         \centering
        \includegraphics[scale=0.4]{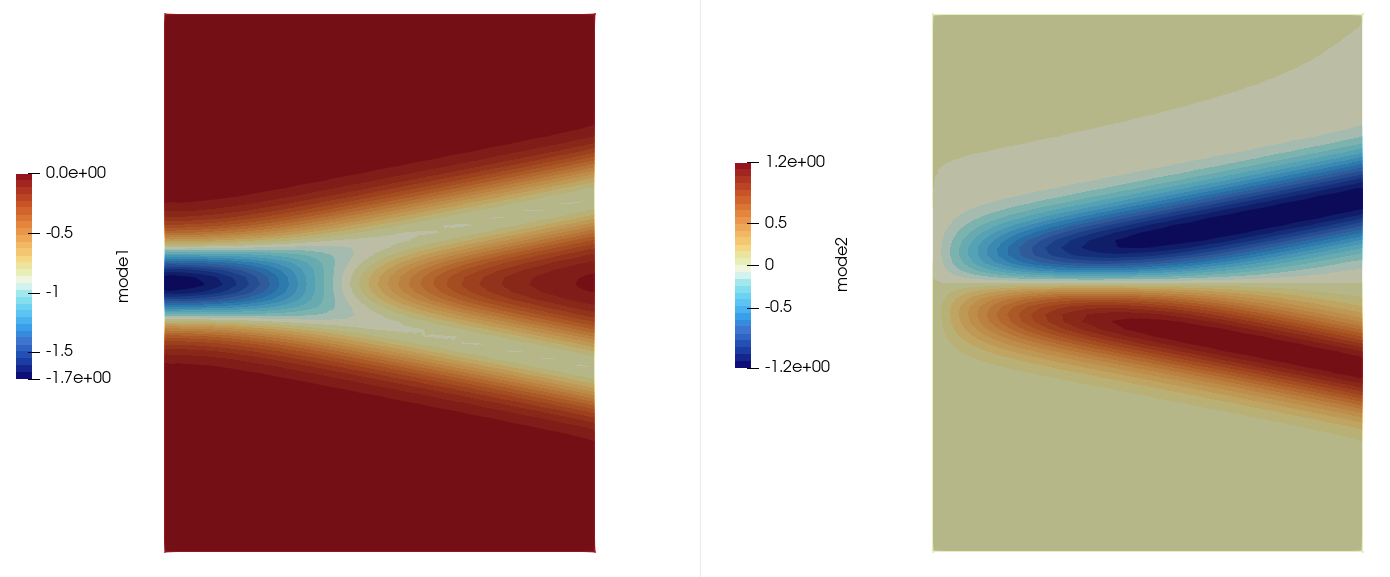}
\caption{\label{fig: POD modes} First two POD modes of the family $\{u_i\circ \bphi_i\}_{1 \leq i \leq n}$ for $r=2$.}
\end{figure}

\subsection{Advection-reaction equation}
This example is taken from \cite{taddei2020registration}. Here, we show the \Ale{benefits using} the optimal morphings on \Ale{the following} advection-reaction problem:
\[
\begin{cases}
\nabla \cdot (c_\mu u_\mu) + \sigma_\mu u_\mu = f_\mu & \text{in } \Omega := (0, 1)^2, \\
u_\mu = u_{D,\mu} & \text{on } \Gamma_{\text{in},\mu} := \{ \boldsymbol{x} \in \partial \Omega : \boldsymbol c_\mu \cdot \bn < 0 \},
\end{cases}
\]
where \(\bn\) denotes the outward \Ale{unit} normal to \(\partial \Omega\), and 
\[
\boldsymbol c_\mu =
\begin{bmatrix}
\cos(\mu_1) \\
\sin(\mu_1)
\end{bmatrix}, \quad
\boldsymbol  \sigma_\mu = 1 + \mu_2 e^{x_1 + x_2}, \quad
\boldsymbol f_\mu = 1 + x_1 x_2,
\]
\[
u_{D,\mu} = 4 \arctan \left( \mu_3 \left( x_2 - \frac{1}{2} \right) \right) (x_2 - x_2^2),
\]
\[
\mu := (\mu_1, \mu_2, \mu_3) \in \mathcal{P} := \left[ -\frac{\pi}{10}, \frac{\pi}{10} \right] \times [0.3, 0.7] \times [60, 100].
\]

We consider $n=250$ samples. In Figure \ref{advection_reaction}, we plot $u_{\mu}$ for three values of the parameter $\mu$ before and after computing the optimal morphings that maximize $J_{r}$ for $r:=1$. 
In this example, the optimal morphings align all the shocks approximately \Ale{at} the same position. 
\begin{figure}[h!] 
     \centering
     \includegraphics[scale=0.55]{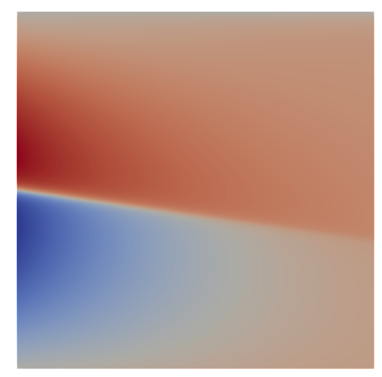}\quad
     \includegraphics[scale=0.55]{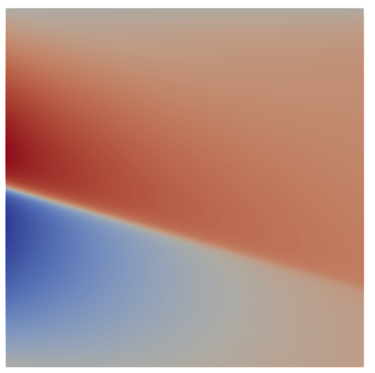}\quad
     \includegraphics[scale=0.55]{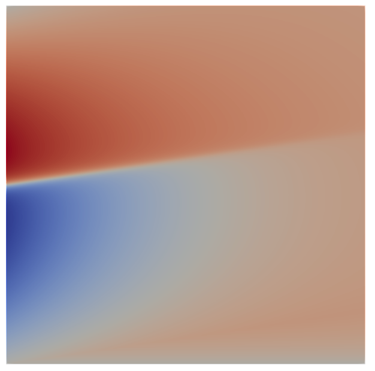}\quad

\vspace{1em}
\hspace{1em}
     \centering
     \includegraphics[scale=0.55]{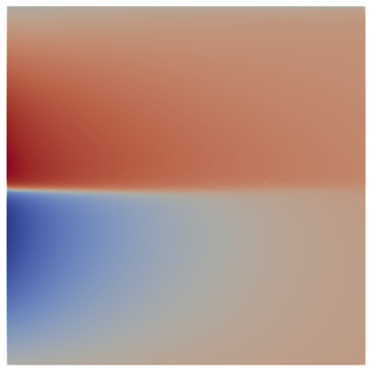}\quad
     \includegraphics[scale=0.55]{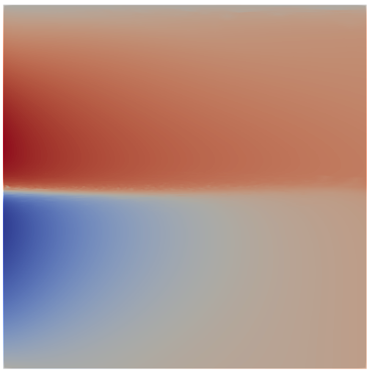}\quad
     \includegraphics[scale=0.55]{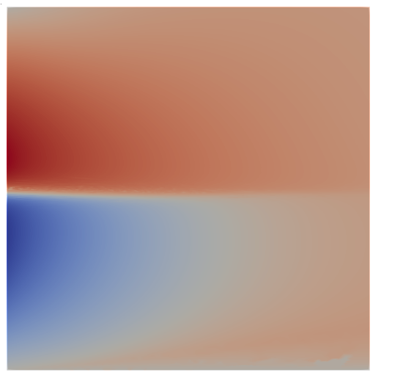}\quad

\caption{ Top: three samples before optimization. Bottom: three samples after \Ale{morphing} optimization.}
\label{advection_reaction}
\end{figure}
The optimal morphing algorithm can be seen as a multi-modal generalization to the registration methods, where the case $r=1$ gives similar results to aligning all the samples on one mode. However, the advantage of our method is that we can go beyond a single mode. Furthermore, the alignment is automatic and does not \Ale{requires} any feature detection \Ale{or} tracking methods.

\FloatBarrier
\section{ O-MMGP: Optimal Mesh Morphing Gaussian Process Regression}
Inspired by MMGP \cite{casenave2024mmgp}, we now present \Ale{a novel regression procedure,} O-MMGP (optimal mesh morphing Gaussian process) regression, which couples the above optimal morphings with Gaussian processes to devise non-intrusive reduced-order models. The main idea is to learn both the morphed fields $(u_{\mu}\circ\bphi_{\mu})_{\mu\in \mathcal{P}}$ and the optimal morphings $(\bphi_{\mu})_{\mu\in \mathcal{P}}$ instead of learning only the poorly reducible fields $(u_{\mu})_{\mu\in \mathcal{P}}$ on itself.

\subsection{Training phase}
In the training phase, we suppose that we have access to the dataset of triplets $(\Omega_i,\mu_i,u_i)_{1\leq i\leq n}$. The domain $\Omega_i$ (or its mesh) and the parameter $\mu_i$ are the inputs to the physical solver, and the field $u_i$ is its output. We assume that all the domains share the same topology. We choose a reference domain $\Omega_0$ which can be one from the dataset. We decompose each morphing $\bphi_i$ into $\bphi_i := \bphi_i^{\rm geo}\circ \bphi_i^{\rm opt}$, where $\bphi_i^{\rm geo}\in \mathbf{M_i}$ (i.e.~$\bphi_i^{\rm geo}(\Omega_0)=\Omega_i$) is the geometric morphing, and $\Phi^{\rm opt}:=(\bphi_i^{\rm opt}:\Omega_0\to \Omega_0)$ solves the problem \eqref{eq: optimization I} that maximizes the compression of the family of functions $(u_i^{\rm ref})_{1 \leq u \leq n} := \{u_i \circ \bphi^{\rm geo}_i\}_{1\leq i\leq n}$. We notice that $\bphi_i^{\rm opt}\in\mathbf{M}_0$, that is, $\bphi_i^{\rm opt}(\Omega_0)=\Omega_0$ for all $1 \leq i \leq n$.

\subsubsection{Preprocessing}
We perform the following preprocessing steps in the training phase. 
\begin{enumerate}

    \item We compute, for all $1\leq i\leq n$, a \Ale{geometric} morphing $\bphi^{\rm geo}_i \in \mathbf{M}_i$. We apply POD to the family $\{\bphi^{\rm geo}_i  -\bId\}_{1\leq i\leq n}$ with a parameter $n_{\rm geo}$ to obtain the POD modes $\{\bzeta^{\rm geo}_i\}_{1\leq i \leq n_{\rm geo}}$  and the generalized coordinates $\{\alpha^i\}_{1\leq i \leq n}$ where $\alpha^i = \left( \alpha^i_j \right)_{1\leq j \leq n_{\rm geo}}\in \mathbb{R}^{n_{\rm geo}}$, such that 
    \begin{align}
  \alpha_j^i= \langle \bphi^{\rm geo}_i -\bId , \bzeta^{\rm geo}_j \rangle_{\boldsymbol{L}^2(\Omega_0)}, \quad \forall 1\leq j \leq n_{\rm geo}, \quad \forall 1 \leq i \leq n .        
    \end{align}

Here, $n_{\rm geo}$ is the number of modes for the geometric morphings. Each domain $\Omega_i$ is defined now by the vector $\alpha^i \in \R^{n_{\rm geo}}$.

\item We solve problem  \eqref{eq: I definition}-\eqref{eq: optimization I} to obtain the functions $\{\bphi^{\rm opt}_i\}_{1\leq i\leq n}$ \Ale{that maximizes compressibility using $r$ modes of the family $\{u\circ\bphi_i^{\rm geo }\circ\bphi_i^{\rm opt }\}_{1\leq i \leq n}$ }. Then, we apply POD to the family $\{\bphi^{\rm opt}_i -\bId\}_{1\leq i\leq n}$ with a parameter $n_{\rm opt}$ to obtain the POD modes $\{\bzeta^{\rm opt}_i\}_{1\leq i \leq n_{\rm opt}}$  and the generalized coordinates $\{\beta^i\}_{1\leq i \leq n}$ where $\beta^i = \left( \beta^i_j \right)_{1\leq j \leq n_{\rm opt}}\in \mathbb{R}^{n_{\rm opt}}$, such that 
\begin{align}
    \beta^i_j= \langle \bphi^{\rm opt}_i  -\bId , \bzeta^{\rm opt}_j \rangle_{\boldsymbol{L}^2(\Omega_0)}, \quad \forall 1\leq j \leq n_{\rm opt}, \quad \forall 1 \leq i \leq n.
\end{align}
  Here, $n_{\rm opt}$ is the number of retained modes for the optimal morphings. 
\item Finally, after evaluating $\{u_i^{\rm opt}\}_{1\leq i\leq n} := \{u_i \circ \bphi^{\rm geo}_i\circ \bphi^{\rm opt}_i\}_{1\leq i\leq n}$, we apply POD with a parameter $r$ to obtain the POD modes $\{\psi_i\}_{1\leq i \leq r}$  and the generalized coordinates $\{\gamma^i\}_{1\leq i \leq n}$ where $\gamma^i = \left( \gamma^i_j \right)_{1\leq j \leq r}\in \mathbb{R}^{r}$, such that $$
\gamma^i_j= \langle \psi_i, u_j^{\rm opt} \rangle_{\boldsymbol{L}^2(\Omega_0)}, \quad \forall 1\leq j \leq r.
$$
\end{enumerate}
By performing the above steps, we obtain the following three approximations for all $1 \leq i \leq n$:
$$ 
\bphi^{\rm ref}_i \approx \,\bId + \sum_{j=1}^{n_{\rm geo}} \alpha^i_j  \bzeta^{\rm ref}_j, \quad \bphi^{\rm opt}_i \approx \,\bId + \sum_{j=1}^{n_{\rm opt}} \beta^i_j  \bzeta^{\rm opt}_j, \quad u_i^{\rm opt} \approx \, \sum_{j=1}^r \gamma^i_j  \psi_j.
$$


\subsubsection{GPR training}

After performing the \Ale{above} dimensionality reduction step, we train two \Ale{GPR} models \cite{rasmussen2006gaussian} as follows. 

\begin{enumerate}
    \item The first model \Ale{aims at} learning the optimal morphing that transforms the reference configuration to the optimal configuration. This model takes as input the physical parameter $\mu_i \in \R^{p}$ and the \Ale{generalized coordinates} $\alpha_i\in \R^{n_{\rm geo}}$ \Ale{of the geometric morphings $\bphi^{\rm geo}$}, and as output the \Ale{generalized coordinates} $\beta_i \in \R^{n_{\rm opt}}$ \Ale{of the optimal morphings $\bphi^{\rm opt}$}. We denote this model by $\mathcal R: \mathbb{R}^p \times \R^{n_{\rm geo}} \to \R^{n_{\rm opt}}$.
    \item The second model \Ale{aims at} learning the field in the optimal configuration. This model takes as input the physical parameter $\mu_i$ and the generalized coordinates $\alpha_i$ \Ale{of the geometric morphings $\bphi^{\rm geo}$}, and as output \Ale{the generalized coordinates} $\gamma_i$ \Ale{of the fields $u_i^{\rm opt}$}. We denote this model by $\mathcal O: \mathbb{R}^p \times \mathbb{R}^{n_{\rm geo}} \to \mathbb{R}^r$.
\end{enumerate}

\subsection{Inference phase}
In the inference phase, we are given a new unseen geometry $\widetilde{\Omega}$, with a physical parameter $\widetilde{\mu}$. The goal is to predict the field $\widetilde{u}$, solution to the physical simulation. We proceed in the following manner.
\begin{enumerate}
    \item First, we compute the geometric morphing $\widetilde{\bphi}^{\rm geo}$ that maps $\Omega_0$ onto $\widetilde{\Omega}$, then we project $\widetilde{\bphi}^{\rm geo}-\bId$ onto the POD basis $\{\bzeta^{\rm geo}_i\}_{1\leq i \leq {n_{\rm geo}}}$ to obtain its \Ale{generalized coordinates} $\widetilde{\alpha}\in \R^{n_{\rm geo}}$.
    \item We use the pair $(\widetilde{\mu},\widetilde{\alpha})$ to infer $\widetilde{\beta}= \mathcal R ( \widetilde{\mu},\widetilde{\alpha}) \in \R^ {n_{\rm opt}}$ and $\widetilde{\gamma}= \mathcal O (\widetilde{\mu},\widetilde{\alpha}) \in \R^{r}$. Then we have 
    \begin{align}
        \displaystyle \widetilde{\bphi}^{\rm opt} := \, \bId + \sum_{j=1}^{n_{\rm opt}} \widetilde{\beta}_j  \bzeta^{\rm opt}_j , \quad \widetilde{u}^{\rm opt} := \, \sum_{j=1}^r \widetilde{\gamma}_j  \psi_j.
    \end{align}    
    \item Finally, the \Ale{field} $\widetilde{u}$ is \Ale{predicted} as $\displaystyle \widetilde{u}:= \widetilde{u}^{\rm opt} \circ \left(\widetilde{\bphi}^{\rm opt}\right)^{-1} \circ \left(\widetilde{\bphi}^{\rm geo}\right)^{-1}.$
\end{enumerate}

\subsection{Multi-Scale Optimization Approach}

We leverage a multi-scale approach to avoid poor local optima. The intuition is that large-scale, coarse features of the fields should be aligned first, before worrying about small details. If we attempt to optimize the morphings $(\bphi_i^{\rm opt})_{1\leq i \leq n}$ using the full-resolution fields  $(u_i)_{1\leq i \leq n}$ directly, the optimization procedure might get stuck while aligning minor features at the expense of major ones. \Ale{To avoid this difficulty}, we construct a smoothed version of the fields, denoted $\widehat{u}_i(c_2)$ (or simply $\widehat{u}_i$ to avoid excessive notation), controlled by a parameter $c_2 > 0$. We define $\widehat{u}_i$ as the solution of a reaction–diffusion equation on $\Omega_0$ with $u_i^{\rm ref}:=u_i\circ\bphi_i^{\rm geo}$ as the source term:

\begin{align}
     -\Delta \widehat{u}_i + c_2\,\widehat{u}_i =c_2 \,u_i^{\rm ref}, \qquad \partial_n \widehat{u}_i|_{\partial \Omega_0} = 0. 
\end{align}
 This equation acts like a low-pass filter: when $c_2$ is small, $-\Delta$ dominates, and $\widehat{u}_i$ is \Ale{smoother (or diffused version) of $u_i^{\rm ref}$.} Instead, when $c_2$ is large, the solution becomes $\widehat{u}_i \approx u^{\rm ref}_i$ (no smoothing). This process being equivalent to convolving the field with a Gaussian kernel (with variance related to $c_2^{-1}$), we refer to it as a convolution filter on the data.
 
We then define a corresponding objective \Ale{functional} $J_r^{(c_2)}[\Phi]$ defined as in \eqref{J definition} but using the smoothed fields $( \widehat{u}_i)_{1\leq i \leq n}$ . For large $c_2$ (minimal smoothing), $J_r^{(c_2)}$ is essentially the original objective functional; for small $c_2$ (heavy smoothing), $J_r^{(c_2)}$ prioritizes alignment of \Ale{large-scale} features.

In practice, we solve a sequence of optimization problems, starting from a heavily smoothed case and progressively reducing the smoothing. At each stage $k$, we find a maximizer, then increase the value of $c_2$, and initialize $J_r^{(c_2)}$ from the previous stage. This \Ale{so-called} continuation method on $c_2$ allows the morphings to continuously deform, following the evolution of an optimum from a coarse alignment toward a fine alignment. 
\begin{rem}[Combining continuation on $c_1$ and $c_2$]
There are many possibilities to combine continuation on $c_1$ and on $c_2$.  A simple strategy is to alternate changing the values of the two parameters. For instance, we can proceed as follows:
\begin{enumerate}
    \item At iteration $2k$, set $c_1^{2k}:=\frac{1}{2}c_1^{2k-2}$.
    \item We start with a small value for $c_2^{0}$. At iteration $(2k+1)$, set $c_2^{2k+1}:=10 \times c_2^{2k-1}$.  
\end{enumerate}
 In our numerical examples, we find that the methodology is not sensitive to the initial choices for both parameters (as expected). In fact, choosing a large value for $c_1$ forces the morphing to be close to the identity map, and and choosing a small value for $c_2$ has a similar effect in terms of convergence.

\end{rem}

\section{Numerical results}
\subsection{Neurips 2024 ML4CFD competion}
In the first example, we apply the \Ale{proposed O-MMGP approach} to the airfoil design case considered in the ML4CFD NeurIPS 2024 competition \cite{yagoubi2024neurips}, and we compare it with the winning solution based on MMGP presented in \cite{yagoubi2025neurips}. We recall that, for each sample in this dataset, we have two scalar inputs, the inlet velocity and the angle of attack, and three output fields, the velocity, pressure and the turbulent kinematic viscosity. The dataset is split into three subsets.
\begin{enumerate}
    \item Training set composed of 103 samples.
    \item Test set composed of 200 samples.
    \item Out-of-distribution (OOD) testing set composed of 496 samples, where the Reynolds number considered for samples in this \Ale{subset} is taken out of distribution.
\end{enumerate}
\Ale{In what follows}, we focus on the turbulent kinematic viscosity field as it represents a poorly reducible field. We illustrate this field in Figure \ref{fig: nut_example}.
\begin{figure}[ht] 
     \centering
     \includegraphics[scale=0.115]{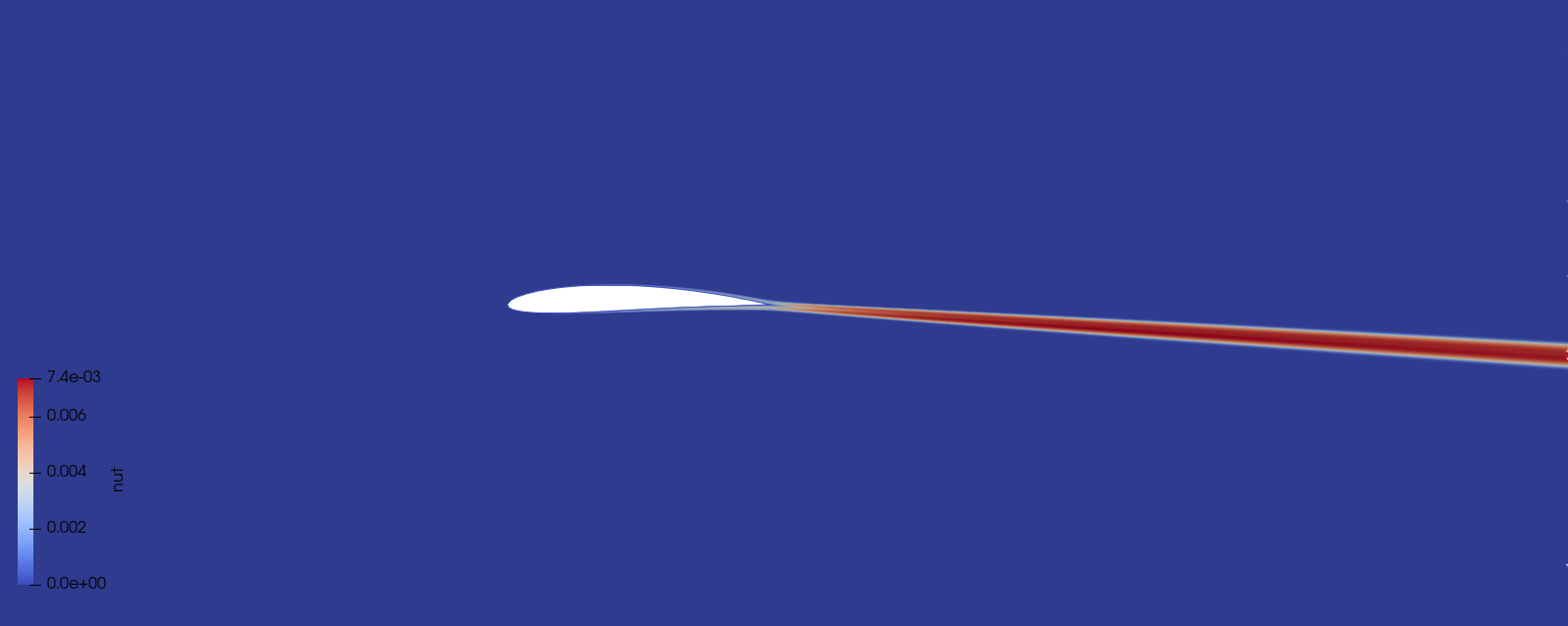}\hfill
     \includegraphics[scale=0.115]{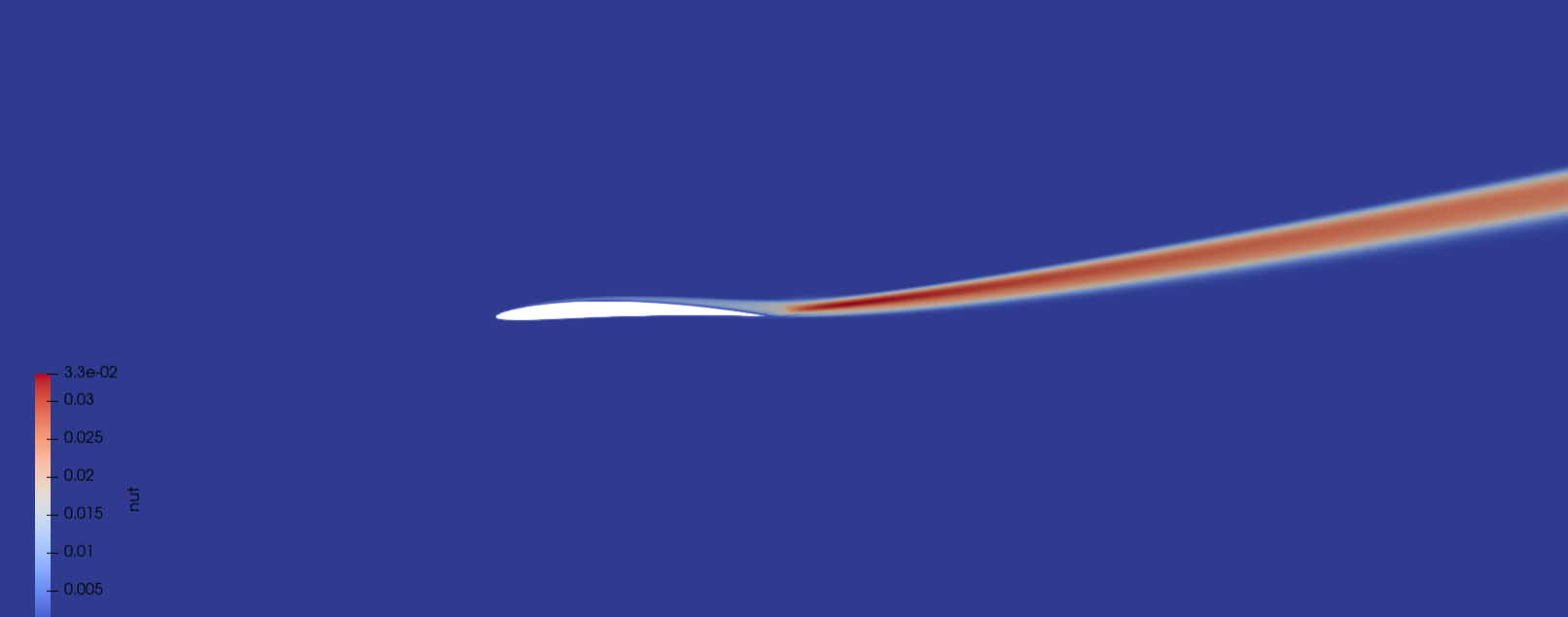}
\caption{ Turbulent kinematic viscosity field illustrated for two of the samples.}
\label{fig: nut_example}
\end{figure}

\begin{figure}[ht] 
     \centering
     \includegraphics[scale=0.1]{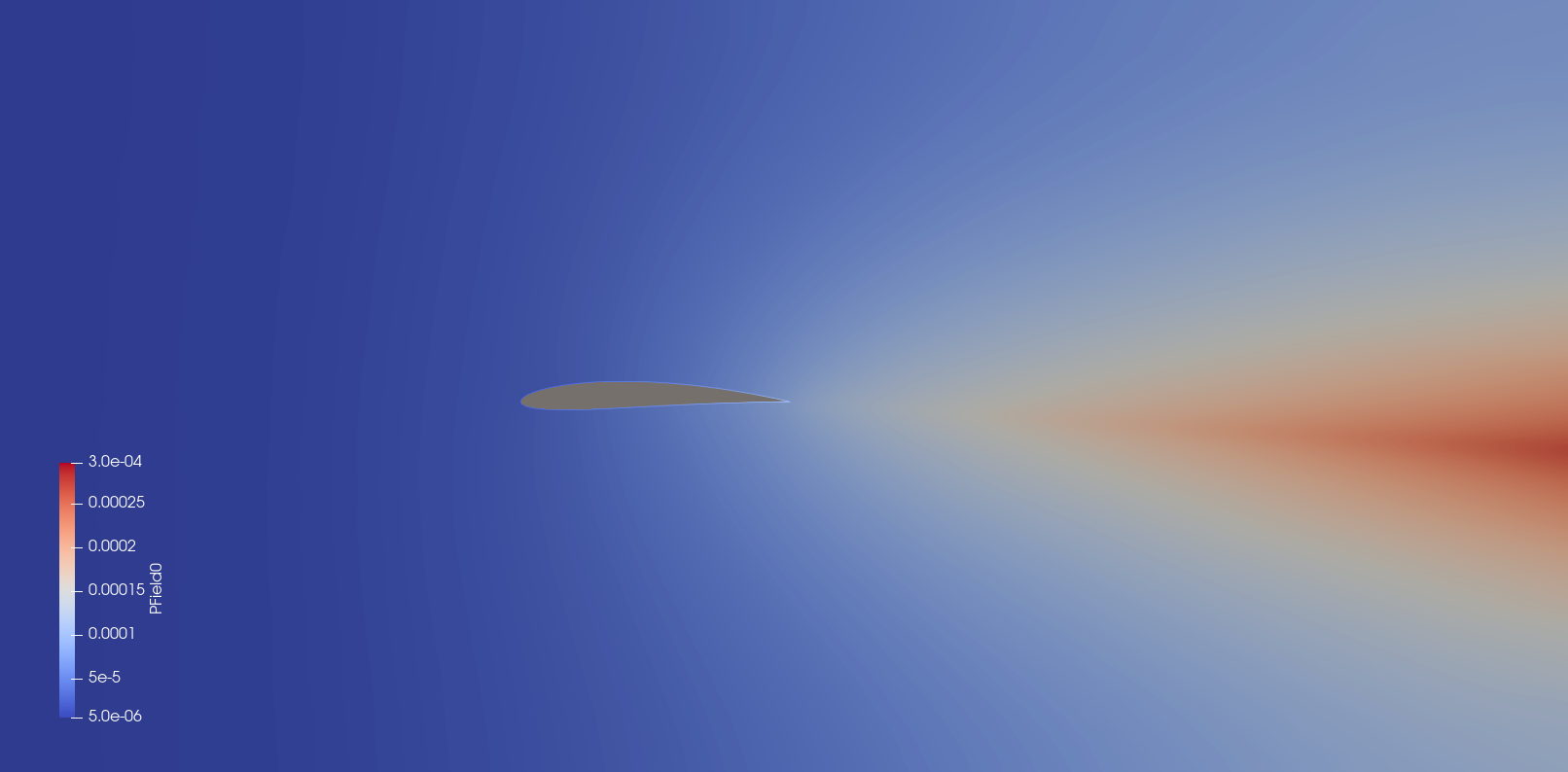}
\caption{The diffused field $\widehat{u}$ for $c_2=1$.}
\label{fig: diffusion ex}
\end{figure}

We run the \Ale{present O-MMGP} algorithm in order to maximize the compression of the turbulent viscosity fields. We choose the reference geometry $\Omega_0$ to be one of the samples in the training set. We then proceed to compute \Ale{geometric morphings} $(\bphi_i^{\rm geo})_{1\leq i \leq n}$ using RBF. \Ale{Then}, we solve the problem \eqref{eq: I definition} for the family $\{\widehat{u}_i(c_2)\}_{1\leq i \leq n}$, using the linear elasticity-based penalty energy $E_{\rm lin}$. Since most of the variation in the fields \Ale{occurs} in the far field and not near the airfoil boundary, we fix the nodes on the airfoil in this example. This allows us to exploit the \Ale{updates} \eqref{gradient update I elasticity} \Ale{and to} assemble the stiffness matrix only once, since the direction of the normal vector on the bounding box does not change through the iterations. We set $r:=1$, $c_1^{(0)}:=0.1$ and $c_2^{(0)}:=1 $. These two latter parameters are changed through the iterations as mentioned \Ale{above}. In Figure \ref{fig: phi opt decay}, we show that the eigenvalues of the correlation matrix of \Ale{optimal morphings} $\{\bphi_i^{\rm opt}\}_{1\leq i \leq n}$ decay rapidly, showing numerically that this family is reducible, and justifying the regression on the optimal morphing POD coordinates. 

\begin{figure}[ht] 
     \centering
     \begin{subfigure}[b]{0.45\textwidth}
         \centering
        \includegraphics[scale=0.4]{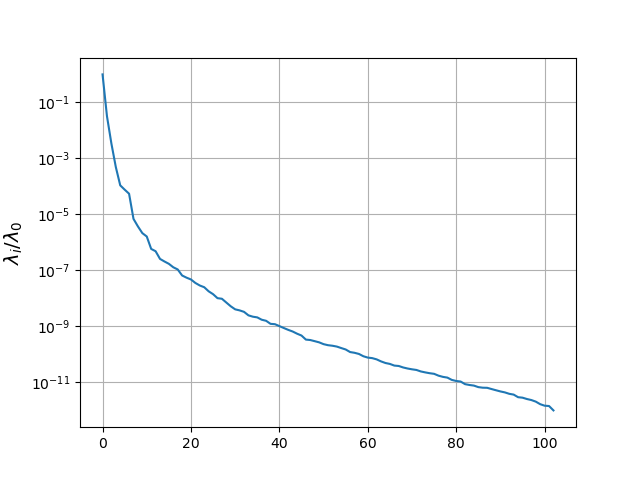}
     \caption{  \label{fig: phi opt decay}
     Decay of the eigenvalues of the correlation matrix of the family $\{\bphi_i^{\rm opt}\}_{1\leq i \leq n}$ .}
     \end{subfigure}
     \hfill
     \begin{subfigure}[b]{0.45\textwidth}
         \centering
\includegraphics[scale=0.4]{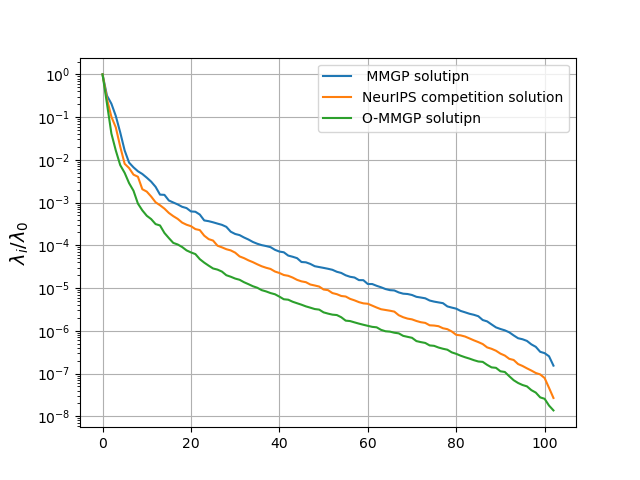}
\caption{\label{fig: nu t decays}
Decay of the eigenvalues of the correlation matrix for the turbulent viscosity fields.}
     \end{subfigure}
\caption{\label{Decay of the eigenvalues} Decay of the eigenvalues.}
\end{figure}

After computing the optimal morphings, we train two Gaussian process regression models, one to learn the optimal morphing $\widetilde{\bphi}$, and the other to learn the field $\widetilde{u}$ (the turbulent viscosity \Ale{obtained}). In order to study the efficiency of the method, we compare the results in the following three \Ale{situations}:
\begin{enumerate}
    \item We apply the method without solving the morphings optimization problem, as in to the original MMGP method. Thus, we predict the turbulent viscosity field in the reference configuration.
    \item We run the winning solution of the NeurIPS 2024 challenge, aligning manually the wake line behind the airfoil.
    \item We apply the O-MMGP procedure described above.
\end{enumerate}
In Table \ref{tab: RMSE nut}, we report the root mean square error (RMSE) of the three tests on the turbulent viscosity field, for the test and OOD sets. While the original MMGP performs poorly for this field, the O-MMGP method produces very accurate predictions, slightly surpassing the NeurIPS 2024 competition solution. The major advantage of \Ale{O-MMGP} is that the alignment of the snapshot \Ale{is} done automatically, without using the specificity of the test case.
\begin{table}[h]
\centering
\begin{tabular}{|c|c|c|c|} 
\hline
 & MMGP & NeurIPS solution & O-MMGP\\
\hline
 Test & 0.143 & 0.025 &  0.024 \\
\hline
 OOD & 0.171   & 0.048 &   0.046   \\
\hline
\end{tabular}
\caption{RMSE errors for the different tests.}
\label{tab: RMSE nut}
\end{table}
In Figure \ref{fig: nu t decays}, we report the decay of the eigenvalues of the correlation matrix for the three methods. For the MMGP solution, no aligning of the fields is performed, which explains the \Ale{poor} decay of the eigenvalues. As expected, the eigenvalues decay most rapidly when using \Ale{O-MMGP}. 

\subsubsection*{Computational cost of the optimal morphing} \label{appendix: cost}

\textbf{Training time:} The training phase of \Ale{O-MMGP} consists of i) computing $\{\bphi_i^{\rm geo}\}_{1\leq i \leq n}$, ii) computing $\{\bphi_i^{\rm  opt}\}_{1\leq i \leq n}$, iii) performing POD and iv) training the Gaussian processes to learn $\bphi^{\rm opt}$ and $u^{\rm opt}$. With respect to MMGP, this adds Step (iii), and one GP training (for $\bphi^{\rm opt}$). For the present test case, steps (i), (iii), and (iv) take about 40 seconds. Step (ii) consists of multiple steps (remeshing the reference mesh to tackle computational bottlenecks, the optimization process, computing $\widetilde{u}(c_2)$ and $\nabla\widetilde{u}(c_2)$ ...). \Ale{Altogether, this step} takes about 35 minutes. We do mention however that some of these steps could be implemented in parallel which would greatly decrease the computational time.\\
\textbf{Inference time:} The inference takes about 81 seconds for the 696 samples (200 from the test set and 496 from the OOD set). Both MMGP and NeurIPS 2024 solutions are faster to train (no optimization problem in the offline pahse), but they are both more expensive than O-MMGP in the online phase. This is due to the fact that both of these methods require the computation of the inverse morphings, as explained in \cite{yagoubi2025neurips}. All the simulations are performed using 128 CPU cores and no GPU. For reference, the full-order solver takes about 25 min for each full simulation.

\subsection{2D\_Profile}
Next, we consider the 2D\_Profile dataset \cite{casenave_2025_15155119}. This dataset consists of 2D steady-state RANS simulations in the supersonic regime with variable geometries that resemble airfoils and propeller blades. The geometries are non-parametrized. The dataset is split into two subsets, a train subset with 300 samples and a test subset with 100 samples. Each sample consists of the computational mesh as an input, and 4 fields of interest as output, which are the Mach number, \Ale{the} pressure and the two components of the velocity field. We show the \Ale{four} fields of two samples in Figure \ref{fig: fields 2dprofile_0}. Most notably, each sample contains a different number of shocks located at different positions, with different orders of magnitude. Furthermore, some samples exhibit a lambda shock (see \Ale{the panels on the second lane} of Figure \ref{fig: fields 2dprofile_0}). We highlight that the variability in the input of this dataset comes from changing the geometry, whereas the inlet, outlet and boundary conditions are kept constant. 
\begin{figure}[h!]
    \centering
         
        \includegraphics[scale=0.2]{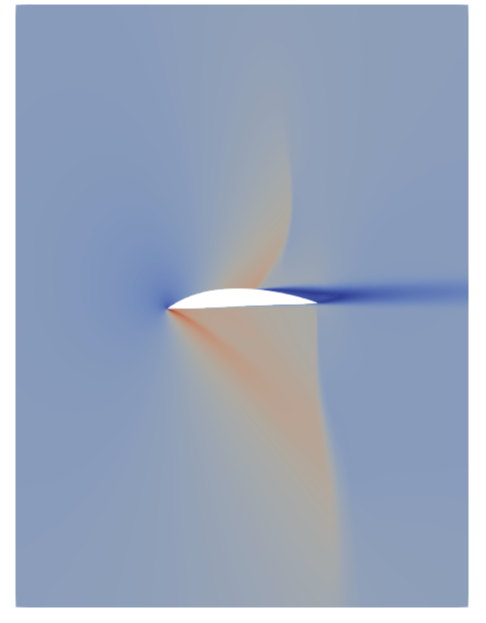} \quad
        \includegraphics[scale=0.2]{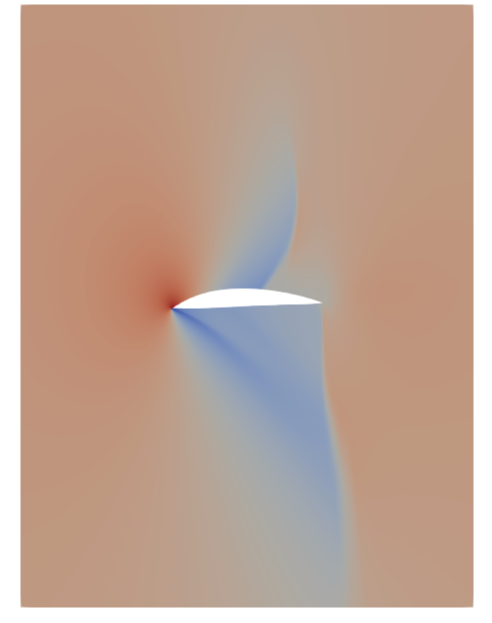} \quad
        \includegraphics[scale=0.2]{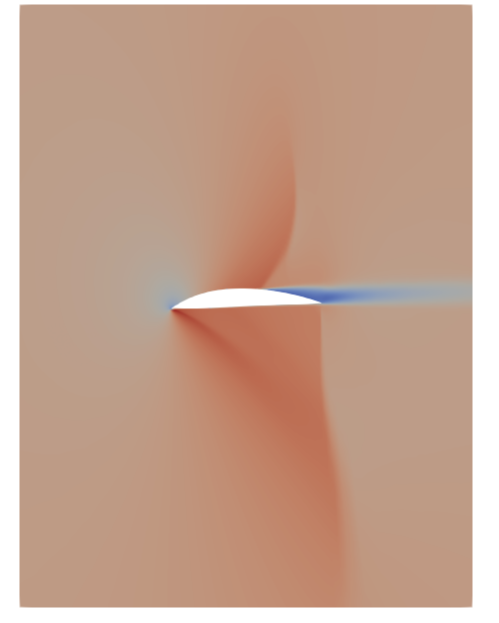} \quad
        \includegraphics[scale=0.2]{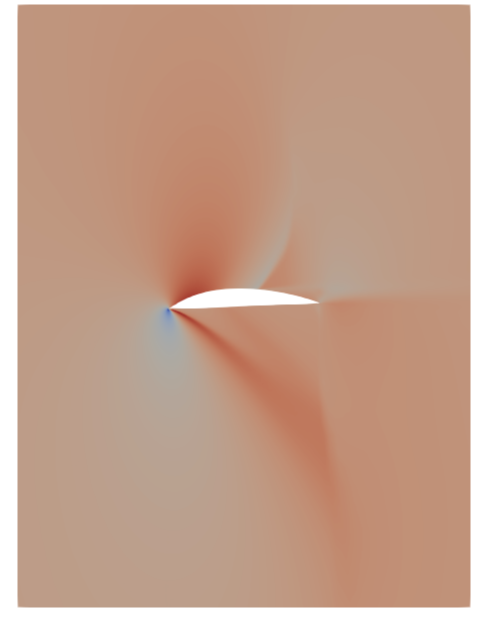} 
        \includegraphics[scale=0.245]{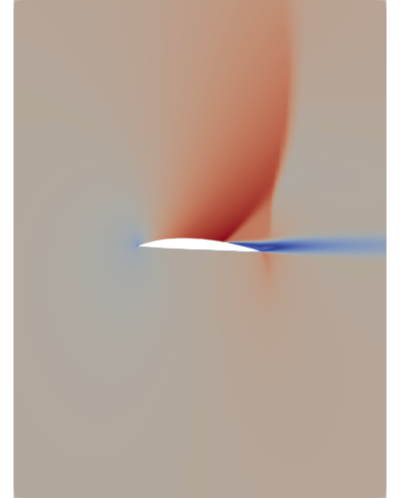} \quad
            \includegraphics[scale=0.245]{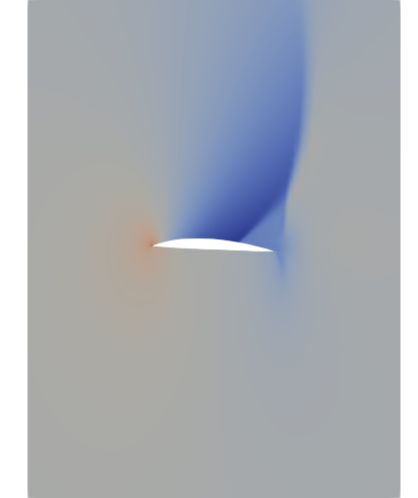} \quad
        \includegraphics[scale=0.245]{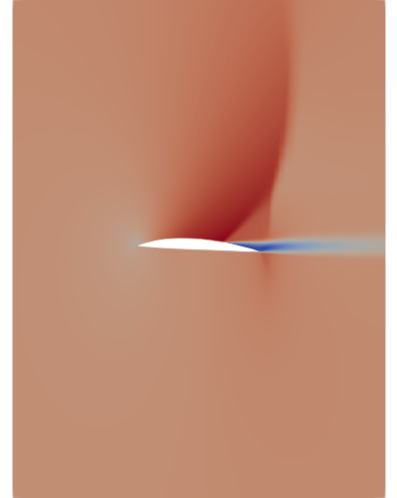} \quad
        \includegraphics[scale=0.245]{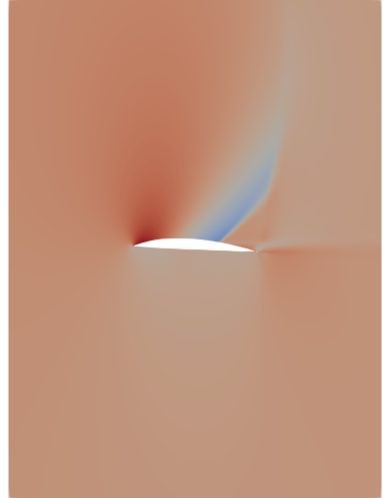}
        
    \caption{The four fields for two samples in the 2D\_Profile dataset. From left to right: Mach number, pressure, $x$-velocity, $y$-velocity.}
    \label{fig: fields 2dprofile_0}
\end{figure}
The computational meshes provided in the dataset are highly anisotropic with different mesh connectivities and number of nodes. We fix one of the samples as a reference domain $\Omega_0$. Since there is no one-to-one correspondence between the meshes, we start by computing a morphing $\bphi_i^{\rm geo}$ from the reference domain onto each target domain $\Omega_i$ in the dataset. The morphings are computed using the elasticity-based morphing technique in \cite{kabalan2025elasticity} with a variable Young modulus $E$ to improve convergence. In order to accelerate the computations, we remesh the reference domain as shown in Figure \ref{fig: refMeshes 2dprofile_0}. \Ale{This} decreases the number of elements from around 70k to 7k. Once the geometric morphings \Ale{are computed}, the morphing optimization algorithm \Ale{is run} same coarse reference mesh. \AK{We test the method for the Mach field and the pressure field.}

\begin{figure}[h!]
    \centering

        \includegraphics[scale=0.25]{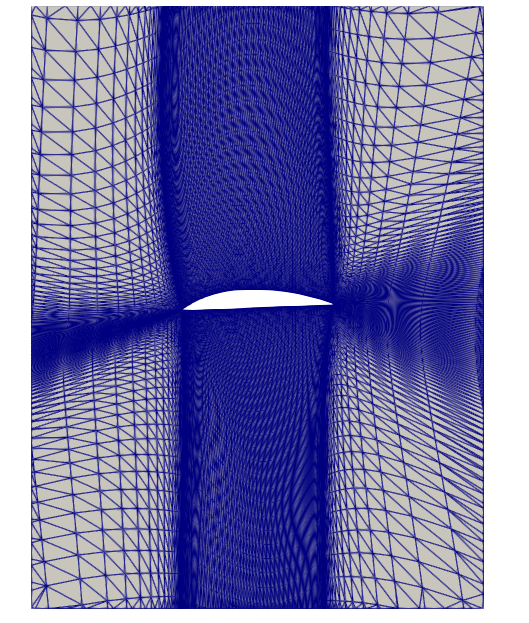} \quad
        \includegraphics[scale=0.25]{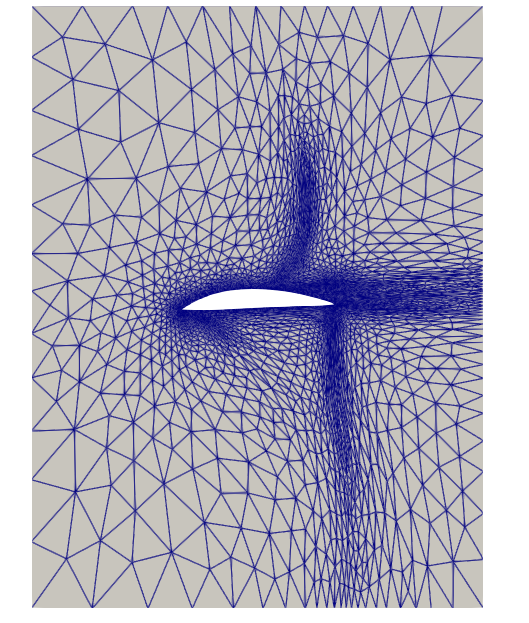} \quad
    \caption{Left: original reference mesh in the dataset. Right: coarse mesh for morphing computation.}
    \label{fig: refMeshes 2dprofile_0}
\end{figure}
\subsubsection{Mach field}
We run the morphing optimization algorithm for the Mach field using the non linear elasticity penalty term $E_{\rm NH}$. We fix $E:=1$ and $\nu:=0.3$ for the linear elasticity bilinear form $a$, used to compute the Riesz representative, $\lambda :=0.5,\mu := 1$ for the non linear penalty term $E_{\rm NH}$, $\epsilon:=0.25$, and $ c_1:=10^{-4}$. We test the method for $s\in\{ 3,4,5,6,8,15,25,30\}$, where $s$ is the number of modes on which we want to project the data. For each value of $s$, we solve

\begin{align} 
    \Phi_s^* \in \arg\max_{\Phi\in \mathbf{M}} I_s[\Phi].
\end{align} 
In Figure \ref{fig: eigenvalues_2dprofile_mach}, we plot $1-J_r[\Phi_s^*]$ as a function of $r$, along with $1-J_r[\Phi^{\rm geo}]$, i.e.~when using only the geometric morphings. We remark that using the optimal morphings is always beneficial in terms of data compression.  
\begin{figure}[h!]
    \centering
        \includegraphics[scale=0.6]{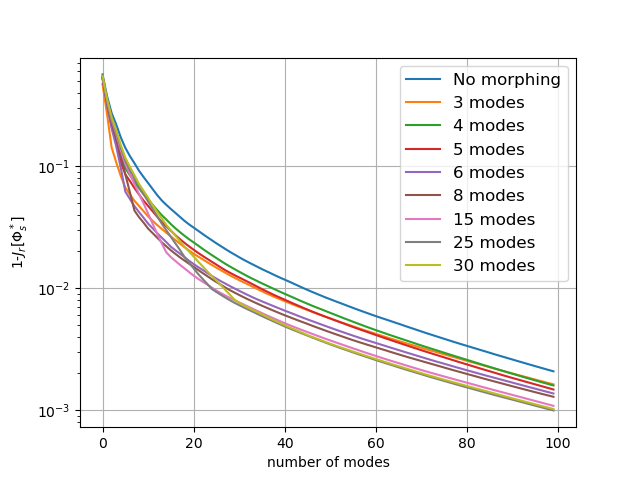}
            \caption{Decay of the eigenvalues of the correlation matrix for the Mach field as a function of the number of POD modes for different values of $s$.}
    \label{fig: eigenvalues_2dprofile_mach}
\end{figure}
Next, we asses the effect of the morphings on the prediction of the Mach field for new (unseen) test samples. We perform POD on $\Phi^{\rm geo}$ and retain $n_{\rm geo}:= 12$ coordinates as inputs to the \Ale{GPR} model. We also perform POD on $\Phi^{\rm opt}$ and retain $n_{\rm opt}:= 32$ modes for the prediction of the optimal morphings for \Ale{the} new samples.  In Figure \ref{fig: 2dprofile_scores_mach}, we report the RMSE for the Mach field on the test set when applying the optimal morphings for different values of $s$, and for different values of the output dimension. From these figures, we can see a clear trend: using the optimal morphings always outperforms using only the geometric morphings. On the other hand, the difference between the different values of $s$ is quite minimal, with $s=6$ giving the best prediction error. We also remark that increasing the output dimension is beneficial up to a certain point where the error start to stagnate. In this case, adding additional modes is no longer beneficial, and only results in increasing the training time.   

\begin{figure}[h!]
    \centering
        \includegraphics[scale=0.45]{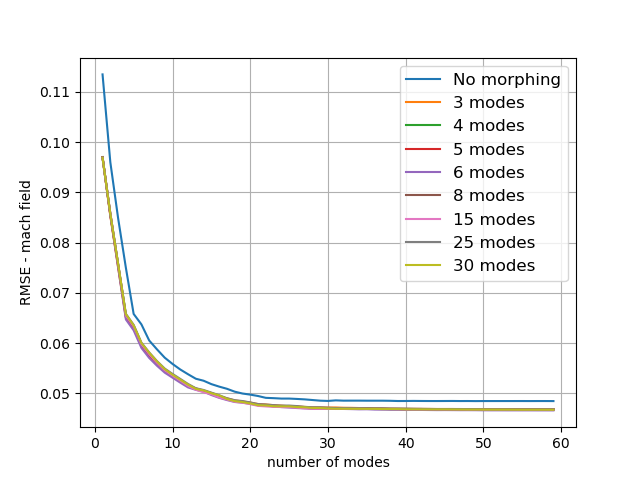} 
        \hfill
         \includegraphics[scale=0.45]{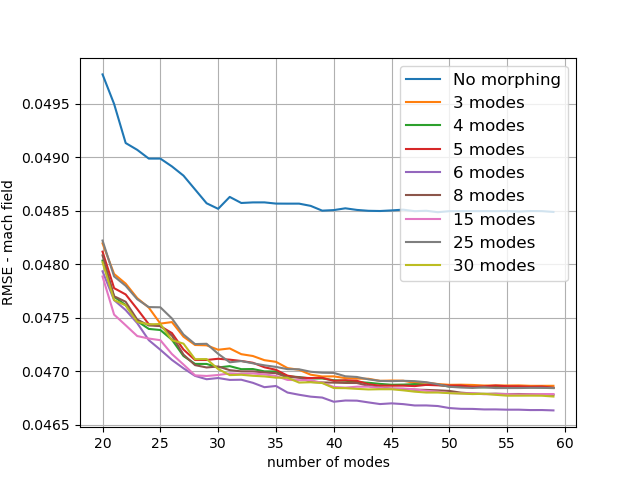}
            \caption{RMSE for the Mach field as a function of the number of POD modes.}
    \label{fig: 2dprofile_scores_mach}
\end{figure}

In Figure \ref{hist: mach error per sample}, we report the number of test samples for different ranges of the prediction error, for both cases with and without optimal morphings. For both cases, most of the samples are almost in the same range. In Figure \ref{fig: test sample 98 - mach}, we plot the true and predicted Mach field for the sample having the highest prediction error. This sample exhibits a lambda shock, which is harder to predict, and which explains the higher prediction error.

\begin{figure}[h!]
    \centering
         
        \includegraphics[scale=0.5]{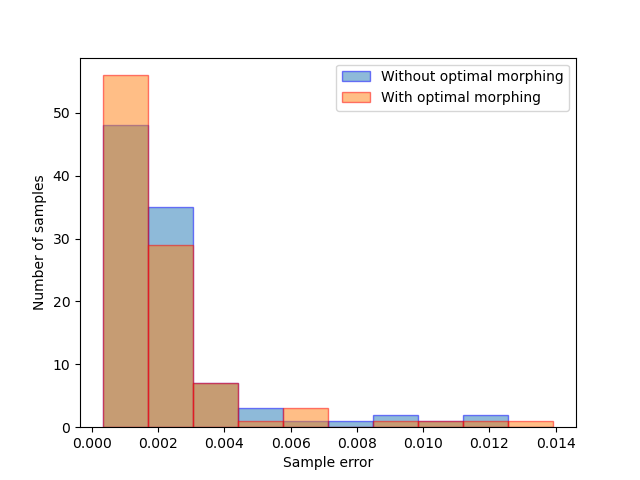} \quad           
    \caption{Number of test samples for different ranges of the prediction error.}
    \label{hist: mach error per sample}
\end{figure}

\begin{figure}[h!]
    \centering

        \includegraphics[scale=0.35]{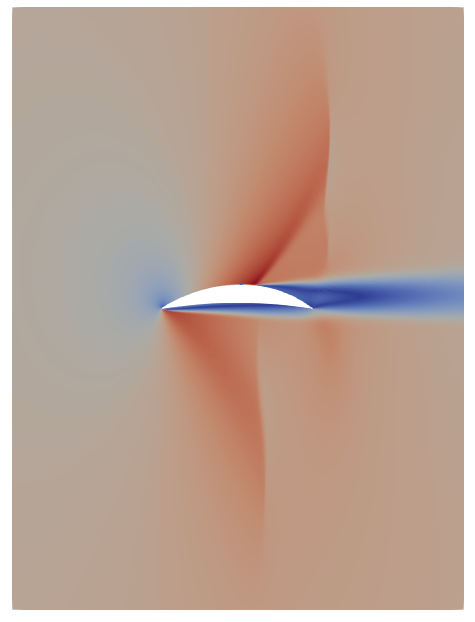} \quad
        \includegraphics[scale=0.35]{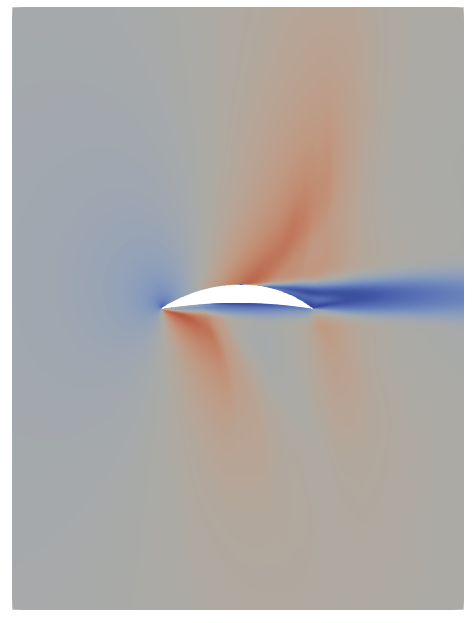} \quad
        \includegraphics[scale=0.35]{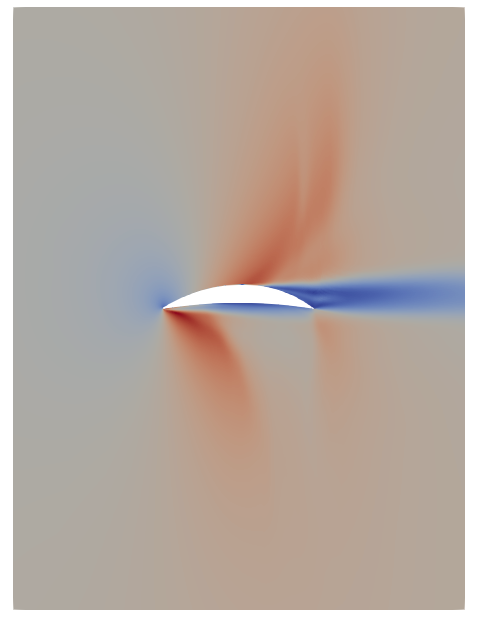} \quad
    \caption{The true and predicted Mach field for the sample having the highest prediction error. Left: True Mach field. Middle: predicted Mach field without optimal morphing. Right: predicted Mach field using the optimal morphing}
    \label{fig: test sample 98 - mach}
\end{figure}

\subsubsection{Pressure field}
We \Ale{now} compute the optimal morphings to maximize the compression of the pressure field. We use the same parameters as before ($E:=1$ and $\nu:=0.3$ for the linear elasticity bilinear form $a$, $\lambda :=0.5, \mu := 1$ for the non linear penalty term $E_{\rm NH}$, $\epsilon:=0.25$, and $ c_1:=10^{-4}$). We test the method for $s \in \{4,6,8,10,15 \}$. In Figure \ref{fig: eigenvalues_2dprofile_pressure}, 
 we plot $1-J_r[\Phi_s^*]$ as a function of $r$, along with $1-J_r[\Phi^{\rm geo}]$. Again, we can see a clear trend: using the optimal morphings is always beneficial in terms of data compression. Next, we plot in Figure \ref{fig: 2dprofile_scores_pressure}, the RMSE for different values of $s$, and for different values of the output dimension. For this field, $s=4$ modes gives the best results with respect to the prediction error. 
 
\begin{figure}[ht]
    \centering
        \includegraphics[scale=0.6]{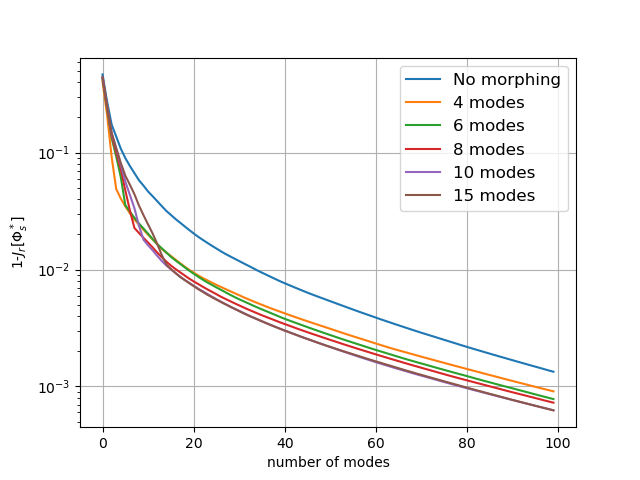}
            \caption{Decay of the eigenvalues of the correlation matrix for the pressure field as a function of the number of POD modes for different values of $s$.}
    \label{fig: eigenvalues_2dprofile_pressure}
\end{figure}

\begin{figure}[ht]
    \centering
        \includegraphics[scale=0.45]{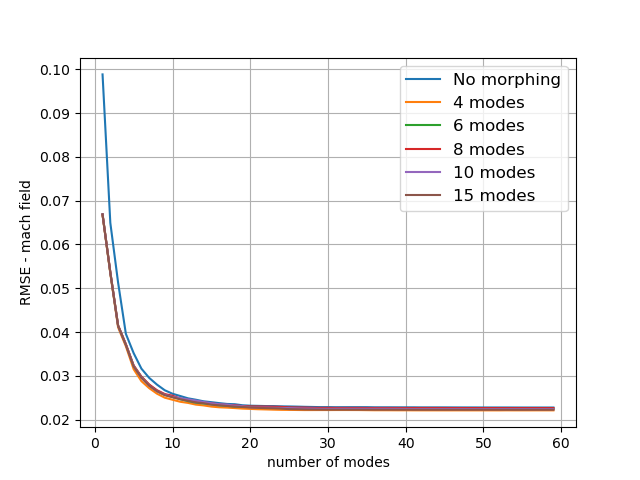} 
        \hfill
         \includegraphics[scale=0.45]{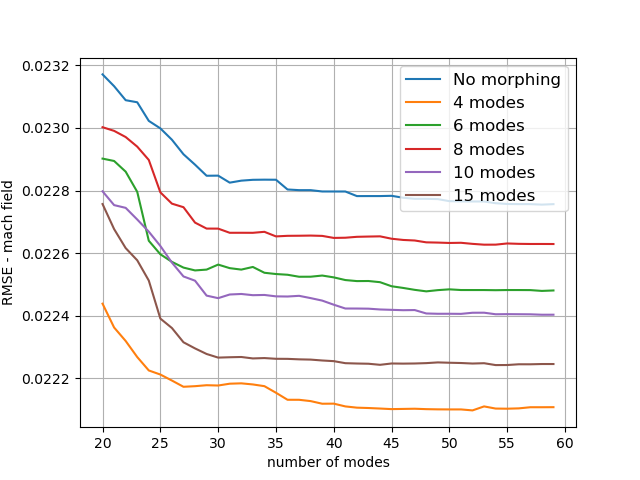}
            \caption{RMSE for the Mach field as a function of number of POD modes.}
    \label{fig: 2dprofile_scores_pressure}
\end{figure}

\FloatBarrier


         



\section{Conclusion}
In this paper, we introduced an optimal morphing framework for model-order reduction targeting poorly reducible problems with geometric variability. Unlike traditional projection-based model-order reduction approaches, which rely on linear subspaces and suffer from the slow decay of the Kolmogorov N-width, the proposed method enhances reducibility by optimally aligning solution fields through morphings. We formulated the morphing computation as a global optimization problem, maximizing the energy captured by the first $r$ POD modes of the transformed snapshots. To solve the optimization problem, we developed a gradient ascent algorithm based on an elasticity-based inner product that ensures smooth and bijective morphings. Additional penalty and continuation strategies were introduced to enforce bijectivity and improve convergence robustness. We tested the method through several numerical examples including synthetic transport problems, advection-reaction equation, and complex CFD datasets such as the AirfRANS and 2D\_Profile cases. The framework demonstrated its capability to enhance data compression, reduce model error, and automate feature alignment without manual intervention. When integrated with GPR in the proposed O-MMGP framework, it enabled enhanced predictions and efficient non-intrusive surrogate modeling across geometrically variable configurations.

\newpage

\bibliographystyle{plain}
\bibliography{bib_thesis}

\end{document}